\documentclass[11pt]{preprint}
\usepackage[full]{textcomp}
\usepackage[osf]{newtxtext}
\usepackage{comment}

\usepackage{dsfont}
\usepackage{amssymb}
\usepackage{mathtools}
\usepackage{hyperref}
\usepackage{breakurl}
\usepackage{mhenvs}

\usepackage{mhequ}
\usepackage{mhsymb}
\usepackage{booktabs}
\usepackage{tikz}
\usepackage{textgreek}
\usepackage{listings}
\usepackage{bm}
\usepackage{orcidlink}

\newcommand{\difffrac}[2]{\mathchoice%
	{\dfrac{#1}{#2}}
	{\frac{#1}{#2}}
	{\frac{#1}{#2}}
	{\frac{#1}{#2}}
}

\usepackage{mathrsfs}
\usepackage{longtable}
\usepackage{graphicx,caption}
\usepackage{subcaption}

\usepackage{float}
\usepackage{microtype}
\usepackage{centernot}
\usepackage{enumitem}
\usepackage{stackrel}
\usepackage[utf8]{inputenc}

\usepackage[T1]{fontenc}
\usepackage{emptypage}
\usepackage{color}

\usepackage{accents}
\newcommand{\ubar}[1]{\underaccent{\bar}{#1}}

\DeclareMathAlphabet{\mcb}{U}{BOONDOX-calo}{m}{n}
\SetMathAlphabet{\mcb}{bold}{U}{BOONDOX-calo}{b}{n}

\DeclareSymbolFont{timesoperators}{T1}{ptm}{m}{n}
\SetSymbolFont{timesoperators}{bold}{T1}{ptm}{b}{n}

\makeatletter
\newcommand*\bcdot{\mathpalette\bcdot@{.5}}
\newcommand*\bcdot@[2]{\mathbin{\vcenter{\hbox{\scalebox{#2}{$\m@th#1\bullet$}}}}}
\makeatother

\makeatletter
\renewcommand{\operator@font}{\mathgroup\symtimesoperators}
\makeatother

\colorlet{symbols}{blue!30!black!50}
\definecolor{purple}{rgb}{0.55,0.05,0.8}

\let\oldskull\skull
\def\skull{\mathord{\oldskull}}

\def\crit{{\mathop{\mathrm{crit}}}}

\DeclareMathAlphabet{\mathbbm}{U}{bbm}{m}{n}

\DeclareFontFamily{U}{BOONDOX-calo}{\skewchar\font=45 }
\DeclareFontShape{U}{BOONDOX-calo}{m}{n}{
	<-> s*[1.05] BOONDOX-r-calo}{}
\DeclareFontShape{U}{BOONDOX-calo}{b}{n}{
	<-> s*[1.05] BOONDOX-b-calo}{}

\setlist{noitemsep,topsep=4pt}

\makeatletter 
\newcommand*{\bigcdot}{}
\DeclareRobustCommand*{\bigcdot}{%
	\mathbin{\mathpalette\bigcdot@{}}%
}
\newcommand*{\bigcdot@scalefactor}{.5}
\newcommand*{\bigcdot@widthfactor}{1.15}
\newcommand*{\bigcdot@}[2]{%
	\sbox0{$#1\vcenter{}$}
	\sbox2{$#1\cdot\m@th$}%
	\hbox to \bigcdot@widthfactor\wd2{%
		\hfil
		\raise\ht0\hbox{%
			\scalebox{\bigcdot@scalefactor}{%
				\lower\ht0\hbox{$#1\bullet\m@th$}%
			}%
		}%
		\hfil
	}%
}
\makeatother

\def\symbol#1{\textcolor{symbols}{#1}}
\def\1{\mathbf{\symbol{1}}}

\usepackage{mathalfa}
\DeclareMathAlphabet\mathzapf{T1}{pzc}{mb}{sc}

\def\eqdist{\stackrel{\text{\scriptsize d}}{=}}
\def\eqdef{\stackrel{\text{\tiny def}}{=}}

\def\init{\mathcal{I}}

\usepackage{stmaryrd}
\def\fancynorm#1{{\talloblong #1 \talloblong}}
\SetSymbolFont{stmry}{bold}{U}{stmry}{m}{n} 

\newcommand{\noise}[1]{\llbracket #1 \rrbracket}
\newcommand{\der}[1]{%
	\left\langle\mkern-3.7mu
	\left[#1\right]
	\mkern-3.7mu\right\rangle%
}

\newcommand{\mrd}{\mathrm{d}}
\newcommand{\mri}{\mathrm{i}}
\newcommand{\floor}[1]{\lfloor #1 \rfloor}

\newcommand{\roof}[1]{\lceil #1 \rceil}
\newcommand{\dd}{(2-d)/2}  

\def\bipar{\sigma}   

\colorlet{darkblue}{blue!90!black}
\colorlet{dgray}{green!20!darkgray}
\colorlet{darkgreen}{green!80!black}
\colorlet{dgreen}{darkgreen!90!cyan}
\colorlet{bgreen}{green!80!blue}
\colorlet{ggreen}{darkgreen!70!gray}

\newcommand{\e}{\varepsilon}

\def\${|\!|\!|}

\def\E{\mathbf{E}}
\def\T{\mathbf{T}}

\newcommand{\mfe}{\mathfrak{e}}

\newcommand{\mft}{\mathfrak{t}}

\newcommand{\mre}{\mathrm{e}}

\newcommand{\mcI}{\mathcal{I}}

\newcommand{\mcD}{\mathcal{D}}

\newcommand{\Var}{\mathrm{Var}}
\newcommand{\Cov}{\mathrm{Cov}}

\def\CP{\mathcal{P}}
\def\CC{\mathcal{C}}
\def\CD{\mathcal{D}}
\def\CI{\mathcal{I}}
\def\CR{\mathcal{R}}

\def\CM{\mathcal{M}}

\def\CT{\mathcal{T}}
\def\CS{\mathcal{S}}
\def\CB{\mathcal{B}}

\def\N{\mathbb{N}}

\def\R{\mathbb{R}}
\def\C{\mathbb{C}}
\def\T{\mathbb{T}}
\def\Z{\mathbb{Z}}
\def\P{\mathbb{P}}
\def\E{\mathbb{E}}

\newcommand{\bmbeta}{\bm{\beta}}
\newcommand{\bmdelta}{\bm{\delta}}

\numberwithin{equation}{section}

\def\dash{\leavevmode\unskip\kern0.18em--\penalty\exhyphenpenalty\kern0.18em}
\def\slash{\leavevmode\unskip\kern0.15em/\penalty\exhyphenpenalty\kern0.15em}

\let\emph\textit 

\makeatother

\colorlet{lbluenode}{blue!25}

\tikzset{
	xi/.style={very thin,circle,fill=lbluenode,draw=symbols,inner sep=0pt,minimum size=1.2mm},
	xib/.style={very thin,circle,fill=lbluenode,draw=symbols,inner sep=0pt,minimum size=1.6mm},
	kernels2/.style={ultra thick,draw=symbols},
}

\makeatletter
\def\DeclareSymbol#1#2#3{%
	\expandafter\gdef\csname MH@symb@#1\endcsname{%
		\tikz[baseline=#2,scale=0.15,draw=symbols,line join=round]{#3}}%
	\expandafter\gdef\csname MH@symb@#1s\endcsname{\scalebox{0.75}{%
		\tikz[baseline=#2,scale=0.15,draw=symbols,line join=round]{#3}}}%
	\expandafter\gdef\csname MH@symb@#1ss\endcsname{\scalebox{0.65}{%
		\tikz[baseline=#2,scale=0.15,draw=symbols,line join=round]{#3}}}%
}
\def\<#1>{\ifmmode\mathchoice{\csname MH@symb@#1\endcsname}{\csname MH@symb@#1\endcsname}{\csname MH@symb@#1s\endcsname}{\csname MH@symb@#1ss\endcsname}\else\csname MH@symb@#1\endcsname\fi}
\makeatother

\DeclareSymbol{Xi}{-2.8}{\node[xib] {};}
\DeclareSymbol{I}{0}{\draw (0,2) -- (0,0);}
\DeclareSymbol{I'}{0}{\draw[kernels2] (0,2) -- (0,0);}
\DeclareSymbol{IXi}{-1}{\draw (0,1) node[xi] {} -- (0,-.4);}
\DeclareSymbol{I'Xi}{-1}{\draw[kernels2] (0,1) node[xi] {} -- (0,-0.4) {};}

\title{
Subcritical non-linear heat equations via spectral gap
}
\author{Ilya~Chevyrev$^1$\orcidlink{0000-0002-5630-9694} and Hora~Mirsajjadi$^{2}$}
\institute{SISSA, Trieste, Italy, \email{ichevyrev@gmail.com} 
	\and {School of Mathematics, The University of Edinburgh,
		United Kingdom, \email{h.s.mirsajjadi@sms.ed.ac.uk}}}

\date{\today}
\begin{document}
	
	\maketitle
	
	\begin{abstract}
We establish local well-posedness for a class of non-linear heat equations on the torus $\mathbb{T}^n$ whose highest order non-linearities have scaling-critical dimension $-1/k, \, k\in \mathbb{N}$, when the initial condition is a random field satisfying a spectral gap inequality involving a suitable Lebesgue norm. The corresponding subcriticality condition matches that of the Gaussian free field in dimension $d<2+2/k$. This extends previous subcritical well-posedness results beyond the Gaussian setting. Moreover, our proof is simpler and avoids diagrammatic arguments.
		\\[.4em]
	\end{abstract}
	\setcounter{tocdepth}{1}
	
	\tableofcontents
	
	\section{Introduction}
	
	Consider the non-linear partial differential equation (PDE)
	\begin{equs}\label{eq:A_eq}
		\partial_t A &= \Delta A + F(A,DA) \qquad &\text{on } (0,T)\times \T^n\;,
		\\
		A_0 &= X \qquad &\text{on } \T^n\;, \label{eq:A_ic}
	\end{equs}
	posed for a function $A\in \CC^\infty ((0,T)\times \T^n;E)$, where $n\geq 1$, $\T^n=\R^n/\Z^n$ is the $n$-dimensional torus, $E$ is a finite-dimensional real inner product space, and
	the initial condition $A_0 = X$ is understood
	in the sense that $\lim_{t\downarrow 0}A_t = X$ in a suitable space of distributions that we make precise below
	(we use the notation $A_t = A(t,\cdot)$).
	Here $\Delta = \sum_{i=1}^n \partial_i^2$ is the Laplacian on $\T^n$,
	$DA\in \CC^\infty ((0,T)\times \T^n,E^n)$, where $E^j$ denotes the $j$-fold Cartesian product of $E$, is the spatial derivative of $A$, i.e. $(DA)_i = \partial_i A$
	and $F\colon E\times E^n\to E$ is a polynomial of the form
	\begin{equ}[eq:F]
		F(x,y) = \sum_{i=0}^{k} Q_i(x,y) + \sum_{j=0}^{2k+1} P_j(x) \;.
	\end{equ}
	In \eqref{eq:F}, $k\ge1$ is a fixed integer.
	For $i\in\{0,\ldots,k\}$ and $j\in\{0,\ldots,2k+1\}$, the maps
	\[
	Q_i \colon E^{i} \times E^n \to E
	\quad
	\textnormal{ and }
	\quad
	P_j \colon E^{j} \to E
	\]
	are, respectively, $(i+1)$-linear and $j$-linear, with $Q_i$ symmetric in its first $i$ arguments and $P_j$ symmetric in all arguments.
	We use the convention that $E^0 = \{0\}$ is the trivial vector space, so that $Q_0\colon E^n \to E$ is linear and $P_0\in E$ is a constant.
	
	We also adopt the shorthand notation
	\[
	Q_i(x,y)
	=
	Q_i(\underbrace{x,\ldots,x}_{i\textnormal{ times}},y)\;,
	\qquad
	P_j(x)
	=
	P_j(\underbrace{x,\ldots,x}_{j\textnormal{ times}})\;.
	\]
	
	\paragraph{Scaling.}
	We denote the non-linearity of highest possible order by
	\[
	F_\crit(x,y)
	\eqdef
	Q_k(x,y)+P_{2k+1}(x)\;,
	\]
	which is critical at scaling dimension
	$-1/k$ in the following sense.
	If we work on $\R^n$ and let $A$ solve \eqref{eq:A_eq}-\eqref{eq:A_ic},
	then the rescaled function
	\[
	A^\lambda(t,x) = \lambda^{\eta} A(\lambda^{-2}t,\lambda^{-1}x)
	\]
	again solves \eqref{eq:A_eq}-\eqref{eq:A_ic} 
	but with initial condition $X^\lambda(x) = \lambda^{\eta} X(\lambda^{-1}x)$ and
	non-linearity
	\[
	F^\lambda(x,y)
	\eqdef
	\lambda^{\eta-2}F\bigl(\lambda^{-\eta}x,\lambda^{1-\eta}y\bigr)
	=
	\sum_{i=0}^{k}\lambda^{-i\eta-1}Q_i(x,y)
	+
	\sum_{j=0}^{2k+1}\lambda^{(1-j)\eta-2}P_j(x)\;.
	\]
	Therefore, taking $\eta=-1/k$, the $F_\crit$ part of $F$ remains unchanged under the scaling, while every lower-order term in \eqref{eq:F} vanishes as $\lambda\to \infty$.
	
	Equivalently, if we assign to $A$ a `scaling dimension' $\eta\in\R$, then $\Delta A$ has dimension $\eta-2$,
	while
	$P_{j}(A)$ and $Q_j(A,DA)$
	have dimensions $j\eta$ and $(j+1)\eta-1$, respectively.
	Thus $F_\crit = Q_k+P_{2k+1}$ is critical at $\eta=-1/k$ in the sense that it has the same scaling dimension as $\Delta A$ (and every
	lower-order term in \eqref{eq:F} is strictly subcritical).
	
	Consequently, the equation is `subcritical' at scaling dimension $\eta$ provided
	\begin{equ}
		\eta-2<(2k+1)\eta \,\Leftrightarrow\, \eta>-1/k
		\quad\textnormal{and}\quad
		\eta-2<(k+1)\eta-1 \,\Leftrightarrow\, \eta>-1/k\;.
	\end{equ}
	In this subcritical regime, the non-linear term is dominated by the linear term $\Delta A$ at small scales.
	Therefore, one can hope to obtain a well-posedness result for \eqref{eq:A_eq}-\eqref{eq:A_ic} in function spaces whose norms shrink\,/\,are invariant under the transformation $f\mapsto \lambda^{\eta}f(\lambda^{-1}\cdot)$ for $\eta>-1/k$ and $\lambda\geq 1$,
	such as the Besov space $B^{\eta}_{\infty,\infty}$.
	
	In the deterministic setting, however, this does not hold.
	Indeed, \eqref{eq:A_eq}-\eqref{eq:A_ic} is well-posed for $X\in B^{\eta}_{\infty,\infty}$ if $\eta > -1/(k+1)$,
	while for $\eta \le -1/(k+1)$, in the generic case that $Q_k$ is not a total derivative, the equation is expected to be ill-posed in the sense of norm inflation.
	(If $Q_k$ is a total derivative and $P_{2k+1}\not\equiv0$, the threshold moves to $\eta \leq -1/(k+\frac12)$.)
	See \cite{Chevyrev22_norm_inf,COW22} and \cite[Sections~2.1-2.2]{Mirsajjadi_26_thesis},
	for proofs of this claim for $k=1$ or the scalar case with $F(x,y)=\pm x^{2k+1}$.
	
	The well-posedness threshold $\eta > -1/(k+1)$ can be seen from the first Picard iterate, which contains the term
	$\int_0^t \mre^{(t-s)\Delta}F(\mre^{s\Delta}X,D\mre^{s\Delta}X)\mrd s$
	in which the integrand in $L^\infty$ can blow up at rate $s^{\frac{(k+1) \eta  -1}{2}}$ due to $Q_k$;
	for $((k+1)\eta -1)/{2} \leq -1 \iff \eta \leq -1/(k+1)$,
	this produces a non-integrable time singularity.
	(If $Q_k$ is a total derivative, one instead has the bound $s^\frac{(2k+1) \eta}{2}$ due to $P_{2k+1}$.)
	
	There is therefore a gap between the scaling subcritical regime $\eta>-1/k$ and the deterministic well-posedness regime $\eta>-1/(k+1)$.
	Given that ill-posedness is a `worst case' result, it is natural to ask whether one can obtain well-posedness in the scaling subcritical regime $\eta>-1/k$ by taking $X$ from a class of random fields.
	This question also has strong motivation from singular stochastic PDEs (SPDEs), where, close to criticality, starting an equation from its natural regularity requires moving beyond the deterministic well-posedness regime,
	see Section~\ref{sec:motivation}.

	\paragraph{Random fields.}
	We are interested in taking $X$ to be a \emph{random} field.
	In light of the above discussion, a natural way to impose ``scaling
	subcriticality'' is to require that there exist $q_*>nk$ such that $X$
	satisfies a spectral gap inequality with $L^{q^*}$-norm
	(see Definition~\ref{def:SG}), where $q^*$ is the H\"older conjugate of
	$q_*$, namely $\frac{1}{q_*}+\frac{1}{q^*}=1$.
	By duality, the corresponding perturbations are measured in $L^{q_*}$.
	Indeed, the $L^q(\R^n)$-norm is invariant under the scaling
	$f\mapsto\lambda^\eta f(\lambda^{-1}\cdot)$
	when $\eta=-n/q$. In particular, the Lebesgue space $L^{nk}(\R^n)$ is invariant under the critical scaling $\eta=-1/k$, whereas $q_*>nk$ corresponds to the
	subcritical scaling dimension $-n/q_*>-1/k$.

	For the Gaussian free field (GFF) $X$ of dimension $d > 2$ (see \cite[Appendix~A]{CM25} for a definition), the natural scaling dimension is
	$\dd$ in the sense that $X$ and $\lambda^{\dd}X(\lambda^{-1}\cdot)$,
	if defined suitably on $\R^n$,
	have the same law, so this subcriticality condition becomes
	\begin{equ}
		\dd>-1/k
		\,\iff\, d<2+2/k\;.
	\end{equ}
	Indeed, the results of Appendix~\ref{app:SG} and \cite[Appendix A]{CM25} imply that the GFF of dimension $d<2+2/k$ on $\T^n$ with $n\geq 2$ satisfies the above spectral gap condition.
	
	\paragraph{Main result.}
	Our main result is an existence and stability result for \eqref{eq:A_eq}-\eqref{eq:A_ic} when we take $X$ to be a random field satisfying the aforementioned spectral gap inequality together with suitable symmetry conditions.
	We collect these conditions in Assumption~\ref{ass:field}
	and state the precise result in Theorem~\ref{thm:main}.
	
	For $a \in \T^n$, let $\tau_a \colon \CD'(\T^n) \to \CD'(\T^n)$ be the translation operator defined by $\tau_a \phi(x) = \phi(x-a)$.
	Let $\CR \colon \CD'(\T^n) \to \CD'(\T^n)$ be the reflection operator $\CR \phi (x) = \phi(-x)$.
	(Both operations are understood using duality.)
	
	\begin{assumption}\label{ass:field}
		Let $n\ge 1$ be an integer and $X$ be an $E$-valued random field on $\T^n$.
		We suppose there exists ${q}_*\in(nk,\infty)$ such that $X$ satisfies the spectral gap inequality with $L^{q^*}$-norm (see \defi{def:SG} below), where ${q}^*$ is the conjugate exponent of ${q}_*$, i.e. $\frac{1}{q_*} + \frac{1}{q^*} = 1$.
		
		We further assume that $X$ is stationary in the sense that, for any $a\in\T^n$,
		\[
		X \eqdist \tau_a(X)\;,
		\]
		where $\eqdist$ denotes equality in law (i.e. in distribution).
		
		Finally, if $F_{\mathrm{crit}}\not\equiv0$,
		then we assume that
		there exists $\bipar\in\{0,1\}$
		such that
		\begin{equ}\label{eq:compatible-symmetry}
			\CR^{\bipar} X\eqdist-X\;,
		\end{equ}
		where $\CR^0=\mathrm{Id}$ and $\CR^1=\CR$, and such that
		\[
		\bipar\equiv k\pmod 2
		\qquad\text{whenever }\; Q_k\not\equiv0\;.
		\]
		If $F_{\mathrm{crit}}\equiv0$, no symmetry assumption is imposed.
	\end{assumption}
	
	In other words, if $Q_k\not\equiv0$, we require symmetry in law
	$X\eqdist-X$ when $k$ is even and odd reflection symmetry
	$\CR X\eqdist-X$ when $k$ is odd. If $Q_k\equiv0$ but
	$P_{2k+1}\not\equiv0$, either of these two symmetries is sufficient.
	These symmetries are used exclusively to establish the probabilistic estimates in Section~\ref{sec:prob-est-SG} (see Lemma~\ref{lem:X^tau,e_mean-zero}).
	
	\begin{theorem}\label{thm:main}
		Let $\chi\in \CC^\infty(\R^n)$ be compactly supported such that $\chi(x)=\chi(-x)$
		and $\int_{\R^n}\chi(x)\,\mrd x =1$, and denote $\chi^\e(x) = \e^{-n}\chi(x/\e)$.
		Suppose Assumption~\ref{ass:field} 
		holds and denote $X^\e\eqdef \chi^\e* X$,
		a mollification of $X$ at scale $\e>0$,
		where we treat $\chi^\e$ as a kernel on $\T^n$ by periodic extension.
		Denote by $A^\e$ the unique maximal smooth solution to \eqref{eq:A_eq}  with initial condition $A^\e_0 = X^\e$.
		
		Then there exists a random variable $T\in (0,1)$ such that, for any $\eta < -n/q_*$
		and $p\in[1,\infty)$,
		\begin{equ}[eq:convergence]
			\lim_{\e\downarrow0} \sup_{t\in(0,T)}\big\{|A^\e_t - A_t|_{\CC^\eta}  + t^{-\eta/2}|A^\e_t - A_t|_\infty\big\} =0
			\;,\qquad
			\E [T^{-p}] < \infty\;,
		\end{equ}
		where the limit holds in $L^p(\P)$ and $\P$-almost surely (a.s.)
		and	where $A$ is smooth on $(0,T)\times \T^n$, solves \eqref{eq:A_eq}, and
		satisfies $\lim_{t\downarrow 0} |A_t-X|_{\CC^\eta} = 0$ almost surely.
	\end{theorem}
	
	Theorem~\ref{thm:main} extends the main result of \cite{CM25} in two ways.
	First, we allow general $k\geq 1$, while \cite{CM25} only considered the case $k=1$, corresponding heuristically to the non-linearity $A^3 + ADA$.\footnote{We note, however, that the methods of \cite{CM25} are sufficiently flexible that the extension to arbitrary $k\in\N$ is straightforward.}
	Second, and much more importantly, we allow the initial condition $X$ to be a general random field satisfying the above spectral gap inequality,
	while \cite{CM25} only considered Gaussian fields on $\T^n$, $n\geq 2$, with covariance $C$ satisfying the bounds
	\begin{equation}\label{eq:CM25_C}
		|C(x)|\lesssim |x|^{-\gamma}
		\;,
		\qquad
		|\nabla C(x)|\lesssim |x|^{-\gamma-1}
	\end{equation}
	for some $\gamma < 2=2/k$.
	(See also \cite[Chapters 2-3]{Mirsajjadi_26_thesis} for an extension of the methods of \cite{CM25} to general $k\geq 1$ and where the bound on $\nabla C$ is relaxed to a $\nu$-H{\"o}lder condition on $C$ with $\nu>0$.)
	As we show in Appendix~\ref{app:SG}, any stationary Gaussian field satisfying just the bound $|C(x)|\lesssim |x|^{-\gamma}$ for $\gamma<\min\{n,2/k\}$ automatically satisfies the spectral gap condition in Assumption~\ref{ass:field}.
	Thus Theorem~\ref{thm:main} recovers the main result of \cite{CM25} as a special case and, in particular, removes the required bound on $\nabla C$ from \eqref{eq:CM25_C}.
	We also give in Appendix~\ref{sec:non-Gaussian} a simple example of a non-Gaussian field (the Steinhaus series) satisfying Assumption~\ref{ass:field} which is not covered by \cite{CM25}.
	
	Perhaps most importantly, the proof of Theorem~\ref{thm:main} is, we believe, significantly simpler than that of \cite{CM25}, even in the Gaussian case.
	The method of \cite{CM25} relies on a careful analysis of Feynman diagrams arising
	from the Gaussian structure, requiring graph theoretic arguments.
	Our present approach is instead entirely analytic and based on the spectral gap inequality, which we apply inductively.

	\begin{remark}
	The only properties of the Laplacian in \eqref{eq:A_eq} that we use are spatial symmetries of the heat kernel $\mre^{t\Delta}$ together with heat flow estimates, see \eqref{eq:heat_flow_estimates} and Lemma \ref{lem:heat-flow-estimates}.
	Our methods apply with little change to other operators satisfying analogous properties, such as the bi-Laplacian $\Delta^2$,
	provided the critical scaling dimension is changed accordingly.
	\end{remark}
	
	\subsection{Method of proof}
	
	We briefly describe the main ideas behind the proof of Theorem~\ref{thm:main}.
	The argument has two main ingredients.
	
	The first is deterministic and is a variant of \cite[Section~2]{CM25}.
	We encode the singular terms appearing in the Picard expansion of
	\eqref{eq:A_eq} by a finite collection of labelled trees.  These terms are
	used to define a non-linear metric space $(\CI,\Theta)$ of rough initial conditions.
	After subtracting the corresponding finite part of the Picard expansion,
	the remainder satisfies an equation which can be solved by a standard
	contraction argument.  This yields a locally Lipschitz solution map on the
	resulting state space; see Theorem~\ref{thm:lwp}.
	The number of
	terms that need to be included increases as the regularity of the initial
	condition approaches the scaling-critical threshold.
	
	The main new ingredient is probabilistic.  We show in
	Section~\ref{sec:prob-est-SG} that, if the initial field satisfies
	Assumption~\ref{ass:field}, then its mollifications converge in the metric
	space $(\CI,\Theta)$, see Theorem~\ref{thm:convergence-of-mollifications}.  The key estimate is obtained by applying the spectral gap
	inequality to each singular term in the Picard expansion.
	Differentiating
	such a term with respect to the initial field $X$ amounts to replacing one of
	the occurrences of $X$ by a deterministic function $h$ of improved regularity.  The resulting
	expressions can then be estimated inductively using heat-flow smoothing and
	previously established bounds.
	Our symmetry assumptions ensure that the most singular terms in the Picard expansion are centred, see Lemma~\ref{lem:X^tau,e_mean-zero}, allowing the spectral gap inequality to
	be applied without introducing renormalisation.
	An induction over the
	trees yields uniform moment estimates as well as estimates for differences
	of mollifications.
	
	We combine in Section~\ref{sec:proof-of-main}
	the deterministic continuity of the solution map with the
	probabilistic convergence estimates to prove Theorem~\ref{thm:main}.

	\subsection{Related works and motivations}
	\label{sec:motivation}
	
	As mentioned above, this work extends the results of \cite{CM25} to general $k\geq 1$ and to non-Gaussian initial conditions.
	The work \cite{CM25} itself extends earlier well-posedness results from \cite{cao2021yang,CCHS22_3D}, which in particular treat GFF-type initial data in 3D,
	to the whole subcritical regime for $k=1$.
	There is also a large literature on using randomness of the initial condition to obtain well-posedness below deterministic regularity thresholds for dispersive equations, including non-linear Schrödinger and wave equations; see, for example, \cite{Bourgain94,Burq_Tzvetkov_08_I,Burq_Tzvetkov_08_II,Burq_10_ICM,Oh_Pocovnicu_16,Sun_Tzvetkov_20,OTW_20_4NLS,DengNahmodYue22,DengNahmodYue24,BringmannDengNahmodYue24} and the references therein.
	The mechanisms in these works are rather different from the one used here,
	but they address related problems of
	probabilistic well-posedness beyond the deterministic threshold.
	
	A second motivation for the present work comes from recent progress in using spectral gap inequalities to establish stochastic estimates underlying BPHZ theorems for singular stochastic PDEs.
	A diagram-free approach based on a spectral gap inequality and Malliavin calculus was developed in \cite{LOTT_24_SG} in the multi-index formulation of regularity structures; see \cite{OttoSauerSmithWeber25,LinaresOttoTempelmayr23,Bruned_Linares_24,BrouxOttoSteele25} and the lecture notes \cite{BOT_26_SG} for background on multi-indexes.
	This method was later adapted to different settings and equations in \cite{HS_24_SG,BailleulHoshino23,BB26_random_models,GT_26_SG}, including the tree-based regularity structures setting, see \cite{Hairer14,BHZ16,BCCH21} and the lecture notes \cite{Chevyrev22_Hopf_lectures} for background.
	See also \cite{chandra2016analytic,hairer2018class} for different approaches to stochastic estimates in regularity structures.
	
	The problem considered here does not require regularity structures since the singularity enters only through the initial condition rather than through a space-time forcing.
	We also do not need to consider renormalisation,
	a central feature in much of the theory of singular SPDEs,
	because our symmetry assumptions are
	compatible with the parity of the leading order non-linearities, $P_{2k+1}(A)$ and $Q_k(A,DA)$,
	and force
	the potentially divergent terms in the Picard expansion to vanish in
	expectation, see Lemma~\ref{lem:X^tau,e_mean-zero}.
	(But see \cite[Section~5.2]{Mirsajjadi_26_thesis} for a discussion of the renormalisation problem in the case of leading non-linearities of the form $A^{2k}$).
	Nevertheless, the stochastic estimates developed here exhibit a similar inductive structure with the spectral gap as the main input.
	A notable difference is that we apply the induction directly to the
	terms in the Picard expansion, in which all products are handled
	classically via heat-flow smoothing, without requiring tools such as
	modelled distributions.
	It would be interesting to understand whether there is a common framework
	encompassing both settings.
	
	There is also a closely related recent literature on reaction-diffusion equations with random initial conditions at or beyond scaling criticality.
	For the Allen--Cahn equation, \cite{Hairer_Le_Rosati_22} studies the evolution from rapidly mixing but vanishing Gaussian initial data in the full subcritical regime and its connection with mean-curvature flow, while \cite{GRZ23} considers two-dimensional white-noise initial data in a weakly critical regime.
	More recently, \cite{CD25McKeanVlasov} studies scaling-critical reaction-diffusion equations with random initial data and obtains McKean--Vlasov limits using Malliavin calculus together with the Gaussian Poincar\'e inequality, and \cite{PY26AllenCahnCLT} proves a central limit theorem for the Allen--Cahn equation with supercritical Gaussian initial conditions using related Malliavin and comparison estimates.
	These works concern scaling limits at critical or supercritical random initial data rather than local well-posedness at fixed coupling in the subcritical regime, but they are closely connected to the questions considered here.
	
	Finally, an important motivation for studying non-linear parabolic equations with rough initial conditions comes from stochastic quantisation and, more generally, flows arising from singular SPDEs.
	As can be seen in, e.g. \cite[Section~2.8]{BCCH21}, \cite[Sections~5-6]{CCHS22_3D},
	solutions to singular SPDEs close to criticality
	take values in H\"older--Besov spaces that are too irregular for even the deterministic (zero noise) equation to be well-posed from these spaces.
	This causes difficulties when trying to restart these dynamics from their own state at a later time.
	The work \cite{CCHS22_3D} handled this problem for the 3D Yang--Mills--Higgs stochastic quantisation equation by constructing a (non-linear) state space together with a corresponding Markov process; see also \cite{cao2024yang_state} for a related state space, \cite{CCHS22_2D} for the 2D case, and the survey \cite{Chevyrev22_YM}.
	It remains an open problem to prove the Markov property for general singular SPDEs close to criticality, such as the $\Phi^4_{4-\delta}$ stochastic quantisation equation for small $\delta>0$ \cite{HS22_Support}.
	
	The deterministic state space constructions of \cite{CCHS22_3D,cao2024yang_state} were
	extended in \cite{CM25} to the full subcritical regime for non-linearities of scaling
	dimension $-1$ (corresponding to $k=1$ here).
	Our construction in Section~\ref{sec:well-posedness} extends
	this further to the general polynomial non-linearities
	considered here (i.e. all $k\geq 1$).
	Together with the probabilistic estimates of
	Section~\ref{sec:prob-est-SG}, which show that random initial
	conditions satisfying a spectral gap inequality belong to these state
	spaces, our results can be seen as a step towards constructing restartable dynamics, and hence Markov
	processes, for broader classes of singular SPDEs.

	\subsection{Notations and preliminaries}\label{sec:notation}

	\textbf{Sets and graphs.} 
	We let $|Z|$ denote the cardinality of a set $Z$.
	We denote $\N = \{1,2,\ldots\}$ and $\N_0 = \N\cup \{0\}$.
	For any $ N\in \N$, we denote $ [N] = \{1,2,\dots, N\} $.
	For $j = (j_1,\ldots,j_n)\in\N_0^n$, we write $|j| = \sum_{i=1}^n j_i$.
	
	We equip $\R^n$ with the Euclidean norm $|\cdot|$ and for $x\in\T^n = \R^n/\Z^n$, we let $|x|$ denote the corresponding geodesic distance of $x$ from $0$ (which is simply the Euclidean norm of the
	corresponding representative if we identify $\T^n$ with
	$[-\frac12,\frac12)^n$ as a set).
	
	All graphs we consider are undirected and finite. 
	A \emph{forest} is a graph without cycles. Every connected component of a forest is called a \emph{tree}.
	
	\medskip 
	
	\textbf{Relations.}
	We use the standard notation $X\wedge Y = \min\{X,Y\}$.
	For $X,Y\geq 0$, we write $X\lesssim Y$ to mean that there exists a constant $K>0$ such that $X\leq KY$. If $X$ and $Y$ are functions, then $K$ is assumed uniform over a given set of variables which is either specified or is clear from the context. If we have both $X\lesssim Y$ and $Y\lesssim X$, then we write $X\asymp Y$.
	
	\medskip
	
	\textbf{Function spaces.} 
	We let $\CC,\CC^\infty,\CD'$ denote the spaces of continuous functions, smooth functions, and (Schwartz) distributions respectively.
	Unless otherwise stated, all function and distribution spaces have domain $\T^n$ and values in $E$, e.g. $\CC^\infty$ means $\CC^\infty(\T^n;E)$.
	
	\medskip
	
	\textbf{Probability.}
	Whenever randomness is involved, all random variables are understood to be
	defined on a probability space $(\Omega,\CF,\P)$,
	and $\E$ denotes expectation with respect to $\P$.
	For $p\in[1,\infty)$, we use the notation
	\[
	\|\cdot\|_{L^p}
	\eqdef
	\bigl(\E|\cdot|^p\bigr)^{1/p}
	\equiv
	\E^{1/p}|\cdot|^p
	\]
	for the probabilistic $L^p$-norm.
	
	\begin{definition}\label{def:heat-kernel}
		For $X\in\CD'$, we denote by
		\begin{equ}
			\CP X \in \CC^\infty((0,\infty)\times \T^n; E)\;,
			\qquad \CP_t X \eqdef \mre^{t\Delta}X\;,
		\end{equ}
		the solution to the heat equation with initial condition $X$.
		Furthermore, for $X\in \CC^\infty([0,T]\times \T^n;E)$
		we denote by
		\begin{equ}[eq:heat_flow_space_time]
			\CP \star X \in \CC^\infty([0,T]\times \T^n; E)\;,
			\quad \CP_t \star X
			\eqdef \int_0^t \mre^{(t-s)\Delta} X_s \,\mrd s\;,
		\end{equ}
		the solution to the inhomogeneous heat equation with source $X$ and zero initial condition\footnote{The notation $\CP_t \star X$ is a shorthand convention, defined by
			\[
			(\CP \star X)_t \;=\; \int_0^t \CP_{t-s} X_s\,\mrd s\;.
			\]
			Although nonstandard, this notation is convenient in the present context.}.
	\end{definition}
	
	For $\eta < 0$ we denote by $\CC^\eta$ the (inhomogeneous) H\"older--Besov space of $E$-valued distributions, defined
	as the Banach space of all $X \in \mcD'(\T^n;E)$  such that
	\begin{equ}\label{eq:CC_eta_def}
		|X|_{\CC^\eta}\eqdef 
		\sup_{\phi \in \CB^r}\sup_{x\in\T^n}
		\sup_{\lambda\in (0,1]} \lambda^{-\eta}|\scal{X,\phi^\lambda_x}| < \infty\;,
	\end{equ}
	where $r=-\floor{\eta}+1$, $\CB^r$ is the set of all  $\phi \in \CC^\infty(\R^n,\R)$ with support in the ball $\{|z|<\frac14\}$ and $|\phi|_{\CC^r}\leq 1$,
	and $\phi^\lambda_x\in\CC^\infty(\T^n,\R)$ is defined by $\phi^\lambda_x (z) = \lambda^{-n}\phi((z-x)/\lambda)$,
	understood as a function on $\T^n$ by periodically extending $\phi$.
	Here and below we write $\floor x$ for the floor of $x\in \R$.
	
	We further denote $\CC^0=L^\infty$, and for $\eta>0$, we let $\CC^{\eta}$ be the usual Banach space of functions with derivatives of order $j\eqdef \roof{\eta}-1$
	(where $\roof{\eta}$ is the ceiling of $\eta$)
	and whose $j$-th order derivatives are $(\eta-j)$-H\"older continuous,
	e.g. $\CC^1$ consists of bounded Lipschitz functions.

	We recall the heat flow estimates for
	$\eta\in \R$ and $\gamma\geq 0$ (see, e.g.~\cite[Lemma~A.7]{gubinelli2015paracontrolled}, or~\cite[Sections~2.1.1-2.1.2, Theorem~2.34]{BookChemin}),
	uniformly in $t\in(0,1)$,
	\begin{equ}[eq:heat_flow_estimates]
		t^{\gamma/2}|\CP_t X|_{\CC^{\eta+\gamma}}\lesssim  |X|_{\CC^\eta}\;.
	\end{equ}
	(Note our convention for $\CC^j$ with $j\in\N_0$ differs from the Besov space $B^j_{\infty,\infty}$ used in \cite{gubinelli2015paracontrolled} but the heat flow estimates remain the same.)
	
	Finally, we use, for a multi-index $j\in \N^n_0$, the shorthand $f^{(j)} = D^j f$.

	\section{Deterministic well-posedness}
	
	\label{sec:well-posedness}
	
	In this section, we introduce a (non-linear) metric 
	space of distributions to which the solution map $X\mapsto A$ for \eqref{eq:A_eq}-\eqref{eq:A_ic}
	extends in a locally Lipschitz way.
	This section is entirely deterministic
	and closely follows \cite[Section~2]{CM25},
	which proves the results of this section for $k=1$.
	We fix an integer $n\ge 1$ and, throughout, use the shorthand notation $Q \equiv Q_k$ and $P \equiv P_{2k+1}$.	
	
	We begin by defining a set $\CT$
	of trees that encode the most singular terms that appear in the Picard iterations for \eqref{eq:A_eq}-\eqref{eq:A_ic}.

	\begin{definition}\label{def:trees}
		For a tree $\tau$ (in the graph theoretic sense), we let $V_\tau$ and $E_\tau$ denote the vertex and edge set respectively of $\tau$ (recall that all graphs we consider are finite).
		\begin{itemize}
			\item A \emph{rooted tree} is a tree $\tau$ with a distinguished vertex $ \rho_\tau \in V_\tau$, called the root.
			\item A \emph{leaf} of a rooted tree $\tau$ is a vertex that either (a) has degree $1$ and is not the root, or (b) has degree $0$ (in which case it is necessarily the root and $\tau$ is just a single vertex).
			\item A \emph{labelled tree} is a rooted tree $\tau$ along with a map $\mfe\colon E_\tau \to \{I,I'\}$,
			where elements of $\{I,I'\}$ are formal symbols which we interpret as labels.	
			We denote by $\CL$ the set of all labelled trees.	
		\end{itemize}
		We denote the rooted tree with a single vertex (the root) by $\<Xi>$.
	\end{definition}

	\begin{definition}[Singular trees]
		\label{def:CT_trees}
		We begin by defining some operations on $\CL$.  
		For $\tau \in \CL$, we let $I(\tau)$ denote the tree obtained by adjoining the root of $\tau$ to a new root via an edge labelled $I$.
		Similarly, we let $I'(\tau)$ denote the tree obtained by adjoining the root of $\tau$ to a new root via an edge labelled $I'$.
		E.g.\ $I(\<Xi>)$ and $I'(\<Xi>)$ are trees with two vertices and one edge labelled $I$ and $I'$ respectively.
		
		For $\tau, \tau' \in \CL$, we define the product $\tau\tau'$ as the tree obtained by merging the roots of $\tau$ and $\tau'$.  
		We view trees in $\CL$ as non-planar rooted trees; in particular, the product $\tau\tau'$ is commutative 
		and associative, consistent with the combinatorial identification of trees up to graph isomorphism (see Remark~\ref{rem:isom}).
		
		\medskip 
		
		The subset $\CT \subset \CL$ of \emph{singular trees} is defined inductively as follows:
		\begin{itemize}
			\item Include $\<Xi> \in \CT$.  
			\item For any $\tau_1,\dots,\tau_{2k+1} \in \CT$, include $\prod_{i=1}^{2k+1} I(\tau_i)$.  
			\item For any $\tau_1,\dots,\tau_{k+1} \in \CT$, include $\big(\prod_{i=1}^{k} I(\tau_i)\big) I'(\tau_{k+1})$.  
		\end{itemize}
		
	\end{definition}
	
	\medskip
	
	For $X\in \CC^\infty$, denote $X^{\<Xi>}=\delta_{t=0}\otimes X \in \CD'(\R\times\T^n;E)$, i.e. the distribution $X$ concentrated on the time $t=0$ hyperplane.	
	We extend notation \eqref{eq:heat_flow_space_time} by writing $\CP_t \star X^{\<Xi>}=\CP_t X \in \CC^\infty([0,\infty)\times\T^n;E)$. 
	For $\tau\in \CT\setminus\{\<Xi>\}$,
	we also define  $X^\tau \in \CC^\infty([0,\infty)\times \T^n;E)$ inductively by
	\begin{equs}[eq:X_tau_def]
		X^{I'(\tau_{k+1})\prod_{i=1}^k I(\tau_i)}_t &= Q(\CP_t \star X^{\tau_1}, \ldots, \CP_t \star X^{\tau_k}, D \CP_t \star X^{\tau_{k+1}})\;,\\
		X^{\prod_{i=1}^{2k+1}I(\tau_i)}_t &= P(\CP_t \star X^{\tau_1},\ldots, \CP_t \star X^{\tau_{2k+1}})\;,
	\end{equs}
	
	\begin{remark}\label{rem:isom}
		Our trees are combinatorial meaning that we do not consider an order for edges leaving a single vertex; more precisely, we identify any two trees which differ by a graph isomorphism preserving roots and labels of edges.
		Note that each term in \eqref{eq:X_tau_def} is well-defined on the isomorphism class of the corresponding (superscript) tree due to symmetry assumptions on $P,Q$.
	\end{remark}
	
	For $\tau\in \CT$, we denote by $\noise{\tau}$ the number of leaves in $\tau$,
	and, for $ i\in\N$, we define
	\begin{equ}
		T_i\eqdef \{\tau\in\CT\,:\; \,\noise{\tau}=i\}\;.
	\end{equ}
	Furthermore, for $ N\in \N$, we denote  
	\begin{equ}
		\CT^N_{\<Xi>} \eqdef \bigcup_{j=1}^N T_j\;,
		\qquad
		\CT^N \eqdef \CT^N_{\<Xi>}\setminus\{\<Xi>\}
		\;.
	\end{equ}
	
	\begin{definition}
		\label{def:norms+init+}
		For a Banach space $(W,|\cdot|)$ and $\delta\in\R$, let $\CC_{\delta}W$  denote the Banach space of continuous functions $f\colon (0,1)\to W$ with norm
		\begin{equ}
			|f|_{\CC_\delta W} \eqdef \sup_{t\in(0,1)}t^{\delta}|f_t|\;.
		\end{equ}
		Let $X,Y\in \CC^\infty$	and 
		$\delta,\beta \in\mathbb{R}$.
		For $\tau \in \CT$ with $\tau\neq\<Xi>$, define the pseudometric
		\begin{align*}
			\fancynorm{X;Y}_{\tau;(\beta,\delta)} &\eqdef |X^\tau-Y^{\tau}|_{\CC_\delta 	\CC^\beta}\;.
		\end{align*}
		For $ N\in \N$, $\omega_{\<Xi>}\in\R$, and vectors $ \bmbeta=(\beta_\tau)_{\tau\in\CT^N}\in \mathbb{R}^{\CT^N} $ and $ \bmdelta=(\delta_\tau)_{\tau\in\CT^N}\in \mathbb{R}^{\CT^N}$,
		define
		\begin{equ}[eq:Theta]
			\Theta_{\omega_{\<Xi>},\bmbeta,\bmdelta}(X,Y)  \eqdef 
			|X-Y|_{\CC^{\omega_{\<Xi>}}}+
			\sum_{\tau\in\CT^N}\fancynorm{X;Y}_{\tau;(\beta_\tau,\delta_\tau)}\,.
		\end{equ}
		Let $\init\equiv\init_{\omega_{\<Xi>},\bmbeta,\bmdelta}$
		denote the completion of smooth functions under the metric $\Theta\equiv \Theta_{\omega_{\<Xi>},\bmbeta,\bmdelta}$.
		We further denote
		$\Theta(X) = \Theta(X,0)$.
	\end{definition}
	
	\begin{definition}\label{def:CI_k}
		Consider $N \in\N$,  vectors $ \bmbeta=(\beta_\tau)_{\tau\in\CT_{}^N}, \bmdelta=(\delta_\tau)_{\tau\in\CT_{}^N}\in \mathbb{R}^{\CT_{}^N}$, and $\omega_{\<Xi>}\le 0$.
		Define, for $\tau\in\CT^N$, $\omega_\tau \eqdef \beta_\tau-2\delta_\tau+2$.
		We say that $N, \omega_{\<Xi>},\bmbeta,\bmdelta$ satisfy condition~\eqref{eq:CI} if
		\begin{equs}\label{eq:CI}
			\begin{aligned}
				&\forall \tau\in\CT_{}^N\,:\quad \beta_\tau \in (-1,0)\;,\quad \delta_\tau < 1\;,
				\quad \omega_\tau\le 0\;, \\
				&\omega \eqdef \min\{\omega_\tau \,:\, \tau\in\CT_{\<Xi>}^N\} > -1/k\;,
				\\
				&\lambda \eqdef \min \Big\{\sum_{i=1}^{k+1}\omega_{\tau_i} \,:\, \tau_i \in\CT_{\<Xi>}^N\,,\; \sum_{i=1}^{k+1}\noise{\tau_i}>N \Big\} > -1\;,
				\\
				&\gamma  \eqdef \min\Big\{\sum_{i=1}^{2k+1}\omega_{\tau_i} \,:\, \tau_i \in\CT_{\<Xi>}^N\,,\; \sum_{i=1}^{2k+1}\noise{\tau_i}>N\Big\} > -2\;.
			\end{aligned}\tag{$\CI$}
		\end{equs}
	\end{definition}
	We remark that, for $\tau\neq\<Xi>$, $\beta_\tau>-1$ and $\omega_\tau\le 0$ together imply $\delta_\tau>1/2$.
	\begin{remark}
		If $Q \equiv 0$ (respectively, $P \equiv 0$), then the condition on $\lambda$ (respectively, $\gamma$) in \eqref{eq:CI} is unnecessary and may be omitted.
	\end{remark}

	The next proposition states that $\mcI$, under condition~\eqref{eq:CI},
	can be realised as a subset of $\CC^{\omega_{\<Xi>}}$. 
	
	\begin{proposition}\label{prop:closable_graph}
		Let $\omega_{\<Xi>}\in\R$, $\bmbeta=(\beta_\tau)_{\tau\in\CT^N}, \bmdelta=(\delta_\tau)_{\tau\in\CT^N}\in \mathbb{R}^{\CT^N}$
		such that $\beta_\tau\in(-1,0)$ and $\delta_\tau<1$ 
		for all $\tau\in\CT^N$. 
		Then the map
		\[
		\CC^\infty \ni X \mapsto \{X^\tau\}_{\tau\in \CT^N} \in (V_{\tau})_{\tau\in\CT^N} \;,
		\]
		has a closable graph in $\CC^{\omega_{\<Xi>}}\times (V_{\tau})_{\tau\in\CT^N}$, where $V_{\tau}=  \CC_{\delta_\tau}\CC^{\beta_\tau}$
		for $\tau\in \CT^N$.
		In particular, $\CI$ is continuously embedded into $\CC^{\omega_{\<Xi>}}$.
	\end{proposition}
	\begin{proof}
		The proof proceeds by a straightforward induction and relies on standard heat flow estimates. See \cite[Proof of Proposition 2.7]{CM25} for a similar argument used to establish an analogous result.
	\end{proof}
	In what follows, assuming condition \eqref{eq:CI}, we denote for $X\in\CI$
	\begin{equ}\label{eq:SN_def}
		\CS^{N}_{\<Xi>}X
		\eqdef
		\sum_{\tau\in\CT^N_{\<Xi>}} c_\tau X^\tau,
		\qquad
		\CS^{N}X
		\eqdef
		\CS^N_{\<Xi>}X-X^{\<Xi>}
		=
		\sum_{\tau\in\CT^N} c_\tau X^\tau\;,
	\end{equ}
	which are well-defined by \prop{prop:closable_graph} since $X^\tau\in V_{\tau}$ for every $\tau\in\CT^N$.
	The combinatorial constants $(c_\tau)_{\tau\in\CT}$ are defined
	recursively as follows. We first set
	\[
	c_{\<Xi>}\eqdef 1\;.
	\]
	
	Consider a tree of the form
	$\tau
	=
	I'(\tau_{k+1}) \prod_{i=1}^{k}I(\tau_i)$.
	Let
	$\{\sigma_1,\ldots,\sigma_{\mft}\}
	=
	\{\tau_1,\ldots,\tau_k\}$
	be the set of distinct trees appearing in the first $k$ branches, and
	define their multiplicities by
	\[
	k_j
	\eqdef
	\big|\{i\in[k]:\tau_i=\sigma_j\}\big|\;,
	\qquad j\in[\mft]\;.
	\]
	Thus,
	$\sum_{j=1}^{\mft}k_j=k$.
	We then define
	\begin{equ}
		c_{I(\tau_1)\cdots I(\tau_k)I'(\tau_{k+1})}
		\eqdef
		\binom{k}{k_1,\ldots,k_{\mft}}
		\prod_{i=1}^{k+1}c_{\tau_i}\;,
		\quad
		\textnormal{where}
		\;\,
		\binom{k}{k_1,\ldots,k_{\mft}}
		=
		\frac{k!}{k_1!\cdots k_{\mft}!}\;.
	\end{equ}
	
	Similarly, consider a tree of the form
	$\tau
	=
	\prod_{i=1}^{2k+1}I(\tau_i)$.
	Let
	$\{\sigma'_1,\ldots,\sigma'_{\mft'}\}
	=
	\{\tau_1,\ldots,\tau_{2k+1}\}$
	be the set of distinct trees appearing among its branches, and define
	\[
	k'_j
	\eqdef
	\big|\{i\in[2k+1]:\tau_i=\sigma'_j\}\big|\;,
	\qquad j\in[\mft']\;.
	\]
	Thus,
	$\sum_{j=1}^{\mft'}k'_j=2k+1$.
	We define
	\begin{equ}
		c_{I(\tau_1)\cdots I(\tau_{2k+1})}
		\eqdef
		\binom{2k+1}{k'_1,\ldots,k'_{\mft'}}
		\prod_{i=1}^{2k+1}c_{\tau_i}\;,
		\quad
		\textnormal{where}
		\;\,
		\binom{2k+1}{k'_1,\ldots,k'_{\mft'}}
		=
		\frac{(2k+1)!}{k'_1!\cdots k'_{\mft'}!}\;.
	\end{equ}
The first $k$ branches in the first recursion are interchangeable because
$Q$ is symmetric in its first $k$ arguments, whereas the branch carrying
the label $I'$ is distinguished. In the second recursion, all $2k+1$
branches are interchangeable because $P$ is symmetric in all its arguments.
	
	
	\begin{theorem}[Well-posedness of \eqref{eq:A_eq}-\eqref{eq:A_ic}] \label{thm:lwp}
		Suppose that  $N\in\N$, $\omega_{\<Xi>}\le 0$, and $ \bmbeta, \bmdelta \in\R^{\CT^N}\!$ satisfy condition~\eqref{eq:CI}.
		Let $\theta> 0$ be such that
		\begin{equ}[eq:theta_assump_k]
			k\omega/2-\theta>-1/2\;,
			\quad
			\lambda/2-\theta>-1/2\;,
			\quad
			\gamma/2-\theta>-1
		\end{equ}
		(such $\theta>0$ exists due to condition \eqref{eq:CI}).
		For $T>0$, let $ \CB_T$ denote the Banach space of functions $R\in\CC([0,T];\CC(\T^n;E))$ for which
		\begin{align*}
			|R|_{\CB_T} \eqdef \sup_{t\in(0,T)} \big\{ t^{-\theta}|R_t|_\infty + t^{\frac12-\theta}|R_t|_{\CC^1} \big\} < \infty\;.
		\end{align*}
		Then there exist $\nu,\e>0$ with the following property.
		For all $K>1$
		and $X\in\init$ such that $\Theta(X)\leq K$,
		if $T^\nu <\e K^{-2k}$,
		then there exists a unique function $R(X)\in \CB_T$ such that
		\begin{equ}
			A\eqdef R(X)+\CP\star \CS^N_{\<Xi>}X \colon (0,T]\to \CC^\infty(\T^n;E)
		\end{equ}
		solves  \eqref{eq:A_eq}
		with initial condition $X$
		in the sense that
		$\lim_{t\downarrow0}|A_t-X|_{\CC^\omega}=0$ where we recall $\omega \eqdef  \min_{\tau\in\CT^N_{\<Xi>}}\omega_\tau$
		and $\CS^N_{\<Xi>}X$ from \eqref{eq:SN_def}.
		
		Furthermore, $ |R(X)|_{\CB_T} \leq K$
		and the map $\{X\in\init \,:\, \Theta(X) \leq K \} \ni X\mapsto R(X)\in \CB_T$ is $1$-Lipschitz.
	\end{theorem}
	\begin{proof}
		For $T\in (0,1)$ and $X\in\init$,
		consider the map $\CM^X\colon \CB_T \to \CB_T$
		\begin{equ}[eq:contraction_mapping_M]
			\CM^X_t(R)
			\eqdef	\CP_t \star F(A,DA) - \CP_t \star \CS^N X =
			\int_0^t \CP_{t-s} F(A_s, D A_s) \,\mrd s
			- \CP_t\star\CS^{N} X\;,
		\end{equ}
		where we denote $A_t = R_t + \CP_t \star \CS^N_{\<Xi>} X$.
		Note that if $R$ is a fixed point of $\CM^X$, then $A$ solves \eqref{eq:A_eq}-\eqref{eq:A_ic} due to the definition of $X^\tau$ and the constant $c_\tau$.
		Standard parabolic regularity then implies that $A\in \CC^\infty((0,T]\times \T^n;E)$.
		
		We show that $ \CM_t^X$ is well-defined, maps $\CB_T$ into itself, and, for $T$ as in the statement, defines a contraction on the ball of radius $K$ of $\CB_T$.
		
		\medskip
		
		Note that $\CP\star F(A, D A) $ is a linear combination of terms of the following form: 
		\begin{enumerate}
			\item $\CP\star Q_i(\CP \star \CS^N_{\<Xi>} X+R, D\CP \star \CS^N_{\<Xi>} X+DR)$, $\,$ with $i=0,\ldots,k$,
			\item $\CP \star P_i(\CP \star \CS^N_{\<Xi>} X+R)$, $\,$ with $\, i=0,\ldots,2k+1$.
		\end{enumerate}
		First, using \eqref{eq:heat_flow_estimates} and the assumptions $\beta_\tau\in (-1,0)$ and $\delta_\tau<1$, for all $\tau\in\CT^N_{\<Xi>}$,
		\begin{equs}
			|\CP_t \star X^\tau|_{\infty}
			&\lesssim
			t^{\frac{\omega_\tau}{2}} \Theta(X) \;,\label{eq:PtX_bound}
			\\
			|D\CP_t \star X^\tau|_{\infty} &
			\lesssim
			t^{\frac{\omega_\tau-1}{2}} \Theta(X) \;.\label{eq:DPtX_bound}
		\end{equs}
		In particular,
		\begin{equ}[eq:PX_DPX_bounds]
			|\CP_t \star \CS^N_{\<Xi>}X|_\infty \lesssim t^{\frac{\omega}{2}} \Theta(X)\;,\qquad
			|D\CP_t\star \CS^N_{\<Xi>}X|_\infty
			\lesssim t^{\frac{\omega-1}{2}} \Theta(X)\;.
		\end{equ}
		
		Next, consider the terms in Cases 1 and 2 that 
		can be written as linear combinations of 
		expressions of the form $\CP_t\star Q(\CP \star X^{\tau_1},\ldots,\CP \star X^{\tau_k}, D\CP \star X^{\tau_{k+1}})$ with $\sum_{i\in[k+1]}\noise{\tau_i}\leq N$ or $\CP_t\star P(\CP \star X^{\tau_1}, \ldots, \CP \star X^{\tau_{2k+1}})$ with $\sum_{i\in[2k+1]}\noise{\tau_i}\leq N$, where $\tau_i \in \CT^N_{\<Xi>}$. 
		These terms are well defined in view of \eqref{eq:PtX_bound}--\eqref{eq:DPtX_bound}, but each such term is cancelled by the corresponding contribution in 
		$-\CP_t \star \CS^N X$ appearing in \eqref{eq:contraction_mapping_M}. 
		Indeed, the term $-\CP_t \star \CS^N X$ in \eqref{eq:contraction_mapping_M} 
		exactly cancels all contributions of this type.
		
		\medskip 
		
		We now analyse the contribution of the remaining terms in the two cases above
		to the right-hand side of \eqref{eq:contraction_mapping_M}.
		
		\textbf{Case 1.} Terms corresponding to $i\neq k$ or containing $j\ge 1$ factors of $R$ are collectively of order
		\begin{equ}
			\tilde C^{(1)}_t \eqdef \sum_{i=0}^{k-1} t^{\frac{(i+1)\omega}{2}+\frac12}\Theta(X)^{i+1}
			+ \sum_{i=0}^{k}\sum_{j\in[i+1]} t^{\frac{(i+1-j)\omega}{2}+j\theta+\frac12}\Theta(X)^{i+1-j}|R|_{\CB_T}^{j}
		\end{equ}
		in $L^\infty$, where we used the assumptions $\omega>-1/k$ and $\theta> 0$ and applied \eqref{eq:PX_DPX_bounds}.
		These terms are also of order $t^{-1/2} \tilde C^{(1)}_t$ in $\CC^1$ due to \eqref{eq:heat_flow_estimates}.
		
		The remaining terms in Case 1 
		can be written as a linear combination of 
		terms of the form $\CP_t\star Q(\CP \star X^{\tau_1},\ldots,\CP \star X^{\tau_k}, D\CP \star X^{\tau_{k+1}})$ with $\tau_i \in \CT^N_{\<Xi>}$ where
		$\sum_{i\in[k+1]}\noise{\tau_i}> N$.
		This contributes 
		$t^{(\lambda+1)/2}\Theta(X)^{k+1}$ in $L^\infty$ and $t^{\lambda/2}\Theta(X)^{k+1}$ in $\CC^1$,
		where we used the assumption $\sum_{i=1}^{k+1}\omega_{\tau_i}\geq \lambda>-1$
		and \eqref{eq:heat_flow_estimates}.
		In conclusion, Case 1 contributes
		\begin{equ}
			C^{(1)}_t \eqdef
			\tilde C^{(1)}_t + t^{\frac{\lambda+1}{2}}\Theta(X)^{k+1}
		\end{equ}
		in $L^\infty$ and $t^{-1/2}C^{(1)}_t$  in $\CC^1$.
		
		\textbf{Case 2.}
		Terms corresponding to $i\neq 2k+1$ or containing $j\ge 1$ factors of $R$ are collectively of order
		\begin{equ}
			\tilde C^{(2)}_t \eqdef \sum_{i=0}^{2k} t^{\frac{i\omega}{2}+1}\Theta(X)^{i}
			+ 
			\sum_{i=1}^{2k+1} \sum_{j\in[i]} t^{\frac{(i-j)\omega}{2}+j\theta+1}\Theta(X)^{i-j}|R|_{\CB_T}^{j}
		\end{equ}
		in $L^\infty$
		and $t^{-1/2}\tilde C^{(2)}_t$ in $\CC^1$,
		where we used $\omega>-1/k$ and $\theta> 0$ and applied \eqref{eq:PX_DPX_bounds}.
		
		The remaining terms in Case 2
		can be written as a linear combination of 
		terms of the form $\CP\star P(\CP \star X^{\tau_1}, \ldots, \CP_t \star X^{\tau_{2k+1}})$ with $\tau_i \in \CT^N_{\<Xi>}$
		where $\sum_{i\in [2k+1]}\noise{\tau_i}> N$.
		The contribution is of order $t^{\gamma/2+1}\Theta(X)^{2k+1}$ in $L^\infty$ and $t^{(\gamma+1)/2}\Theta(X)^{2k+1}$ in $\CC^1$,
		where we used the assumption $\sum_{i=1}^{2k+1}\omega_{\tau_i} \geq \gamma  > -2$
		and \eqref{eq:heat_flow_estimates}.
		In conclusion, Case 2 contributes
		\begin{equ}
			C^{(2)}_t \eqdef \tilde C^{(2)}_t +
			t^{\frac{\gamma}{2}+1}\Theta(X)^{2k+1}
		\end{equ}
		in $L^\infty$ and $t^{-1/2}C^{(2)}_t$ in $\CC^1$.
		
		\medskip
		
		Combining the above cases, we have the estimate
		\begin{equ}[eq:M_inf]
			t^{-\theta}|\CM^X_t(R)|_\infty
			\lesssim t^{-\theta} \sum_{i=1}^2 C^{(i)}_t 
			\lesssim t^\nu (1+\Theta(X)^{2k+1} + |R|_{\CB_T}^{2k+1})\;,
		\end{equ}
		where
		\begin{equ}
			\nu \eqdef \min\{(k\omega+1)/2,\,(\lambda+1)/2,\,\gamma/2+1\}-\theta 
		\end{equ}
		(we used here the estimates on $C^{(i)}_t$ above and the assumptions $\theta> 0 \ge \omega$).
		The same estimate holds for $t^{-\theta+1/2}|\CM_t^X(R)|_{\CC^1}$, so in conclusion
		\begin{equ}
			|\CM^X(R)|_{\CB_T} \lesssim T^\nu (1+\Theta(X)^{2k+1} + |R|_{\CB_T}^{2k+1})\;.
		\end{equ}
		We remark that $\nu>0$
		due to \eqref{eq:theta_assump_k}.
		It follows that there exists $\e>0$ such that, for all $K>1$ and $X\in\init$ with $\Theta(X)\leq K$,
		if $T^{\nu}\leq \e K^{-2k}$,
		then $\CM^X$ stabilises the ball $\{R\in\CB_T\,:\,|R|_{\CB_T}\leq K\}$.
		
		\medskip
		
		Furthermore, for another $\bar X\in\init$ with $\Theta(\bar X)\leq \Theta(X)$ and $\bar R\in\CB_T$,
		almost the same considerations imply that
		\begin{equs}[eq:M_X_diffs]
			|\CM^X(R) - \CM^{\bar X}(\bar R)|_{\CB_T}
			&\lesssim T^\nu \Theta(X,\bar X)(1+\Theta(X)^{2k})
			\\
			&\quad +
			T^\nu |R-\bar R|_{\CB_T}(1+\Theta(X)^{2k}+|R|_{\CB_T}^{2k} + |\bar R|_{\CB_T}^{2k})\;.\quad
		\end{equs}
		Taking $\bar X=X$, we obtain that $\CM^X$ is a contraction on the ball $\{R\in\CB_T\,:\,|R|_{\CB_T}\leq K\}$
		whenever $T^\nu < \e K^{-2k}$ for $K>1$ and $\Theta(X)\leq K$ and $\e>0$ sufficiently small.
		It follows from Banach's fixed point theorem that there exists a unique fixed point $R$ to $\CM^X$ in this ball.
		Uniqueness in all of $\CB_T$ follows in a standard way by restarting the equation.
		The claimed Lipschitz estimate $|R-\bar R|_{\CB_T}\leq \Theta(X,\bar X)$ for $\eps$ sufficiently small follows from
		taking $R$ and $\bar R$ in \eqref{eq:M_X_diffs} as the unique fixed points.
		
		It remains to prove that $\lim_{t\downarrow0}|A_t-X|_{\CC^\omega}=0$, which follows exactly as the analogous claim in the proof of \cite[Theorem~2.9]{CM25}.
	\end{proof}
	
	\section{Probabilistic estimates}
	\label{sec:prob-est-SG}
	In this section, we show that, for any random field $X$ satisfying \assu{ass:field}, there exist $N\in\N$, $\omega_{\<Xi>}\in\R$, and $\bmbeta,\bmdelta\in\R^{\CT^N}$ such that mollifications $X^\e$ converge in the state space $\init_{\omega_{\<Xi>},\bmbeta,\bmdelta}$ almost surely
	and in $L^p(\P)$ for all $p\ge 1$.
	We make this statement precise in \theo{thm:convergence-of-mollifications}.
	\begin{definition}\label{def:homogeneity}
		Consider $q_*,n$ as in \assu{ass:field} and denote 
		\begin{equ}
			\alpha = -n/q_*\;.
		\end{equ}
		Let $\CL$ be the set of labelled trees as in \defi{def:trees}. Define $|\cdot| \colon \CL\to\R$ by
		\begin{equ}[eq:homogeneity]
			|\tau| = (\alpha-2) \noise{\tau} + 2|E_\tau|-|E'_\tau|\;, \qquad \tau\in\CL\;,
		\end{equ}
		where we denote by $\noise{\tau}$ the number of leaves of $\tau$ and $E'_\tau \eqdef \{e\in E_\tau :\, \mfe(e)=I'\}$. ($|\cdot|$ on the right-hand side denotes the cardinality of the input set.)
	\end{definition}
	\begin{remark}\label{rem:homogeneity}
		A straightforward induction implies that $|\tau| = \noise{\tau}\big(\alpha +1/k\big)- 2-1/k$ for all $\tau\in\CT$.
	\end{remark}
	
	\begin{theorem}\label{thm:convergence-of-mollifications}
		Suppose that $X$ satisfies \assu{ass:field}. Let $N\in\N$,
		$\omega_{\<Xi>}<\alpha$,
		and, for $\tau\in\CT^N$,
		$\beta_\tau \in [\alpha,0)$
		and $\delta_\tau>-|\tau|/2+\beta_\tau/2$.
		Consider the mollifications $X^{\e}$ defined in \theo{thm:main}
		and denote $\Theta = \Theta_{\omega_{\<Xi>},\bmbeta,\bmdelta}$ where $ \bmbeta\eqdef(\beta_\tau)_{\tau\in\CT^N}$ and $ \bmdelta\eqdef(\delta_\tau)_{\tau\in\CT^N}$.
		Then there exists $\kappa>0$ such that, for all $p\in[1,\infty)$,
		\[
		\E \Big|\sup_{0 < \bar\e < \e \leq 1} \frac{\Theta(X^\e,X^{\bar\e})}{|\e-\bar\e|^\kappa}\Big|^p < \infty\;.
		\]
		In particular, $X^\e$
		converge as $\e\downarrow 0$ in $(\init_{\omega_{\<Xi>},\bmbeta,\bmdelta},\Theta_{\omega_{\<Xi>},\bmbeta,\bmdelta})$ in $L^p(\P)$ for all $p\in [1,\infty)$ and $\P$-almost surely.
	\end{theorem}
	
	\begin{remark}
		The exponent $\delta_\tau$ in Theorem~\ref{thm:convergence-of-mollifications}
		is allowed to be negative and this causes no issues with respect to the definition of the
		state space $\CI$.
		Indeed, the smooth functions
		$X^\tau_t$ are of the correct order as $t\downarrow0$ due to the condition $    \delta_\tau>-|\tau|/2+\beta_\tau/2$.
		For instance, for the first-generation
		trees $\tau= I'(\<Xi>) \prod_{i=1}^k I(\<Xi>), \, \prod_{i=1}^{2k+1} I(\<Xi>)$,
		this condition automatically forces $\delta_\tau>0$
		because for these trees, we have
		$|\tau|=(k+1)\alpha-1$ and $|\tau|=(2k+1)\alpha$ respectively, and
		since $\beta_\tau\geq\alpha$ and $\alpha<0$,
		\[
		-|\tau|/2+\beta_\tau/2
		\geq \frac{1-k\alpha}{2}>0\;,
		\quad\text{respectively}\quad
		-|\tau|/2+\beta_\tau/2
		\geq -k\alpha>0\;.
		\]
		Thus negative values of $\delta_\tau$ can arise only for `deeper trees', which indeed
		vanish at $t=0$, and the order of decay is consistent with the lower bound on $\delta_\tau$.
	\end{remark}
	
	To prove \theo{thm:convergence-of-mollifications}, we use the following definitions and lemmas.
	
	\medskip
	
	We say that a function $F\colon \CD'(\T^n; E)\to \R$ is \emph{cylindrical} if there exist $m\in\N$, $\phi_1,\ldots,\phi_m \in \CD(\T^n;E)$, and a smooth function $f\colon \R^m\to \R$ such that $F(\xi)=f(\xi(\phi_1),\ldots,\xi(\phi_m))$.
	In this case we define the functional derivative $\frac{\delta F}{\delta \xi}\colon \CD'(\T^n; E)\to \CD(\T^n; E)$ by
	\begin{equ}
		\frac{\delta F}{\delta \xi}[\xi] = \sum_{i=1}^m \d_i f(\xi(\phi_1),\ldots,\xi(\phi_m))\,\phi_i\;.
	\end{equ}
	Throughout, for brevity, we write $\frac{\delta F}{\delta \xi}$ instead of $\frac{\delta F}{\delta \xi}[\xi]$.
	
	\begin{remark}
		The functional derivative $\frac{\delta F}{\delta \xi}$ introduced above is well-defined; that is, it does not depend on the chosen representation of $F$. 
		Indeed, one can characterise it by the G\^ateaux derivative identity
		\begin{equ}
			\bigl\langle \tfrac{\delta F}{\delta \xi}, h \bigr\rangle
			= \lim_{\e \downarrow 0} 
			\difffrac{F(\xi + \e h) - F(\xi)}{\e}\;,
		\end{equ}
		for every $h \in \mathcal{D}'(\T^n; E)$,
		so that the right-hand side depends only on $F$ itself.
	\end{remark}
	
	\begin{definition}
		\label{def:SG}
		Let $\|\cdot\|$ be a norm on $\CD(\T^n;E)$. 
		We say that a random Schwartz distribution $\xi \in \CD'(\T^n;E)$ 
		satisfies the 
		\emph{spectral gap inequality} 
		with norm $\|\!\cdot\!\|$ if 
		\begin{equ}[eq:SG-inequality]
			\E^{\frac12}\big|F(\xi)\big|^2 
			\,\lesssim\,
			\big|\E[F(\xi)]\big| 
			+ \E^{\frac12}\Big\|\frac{\delta F}{\delta \xi}\Big\|^2\,,
			\tag{$\mathrm{SG}_2$}
		\end{equ}
		for all bounded cylindrical functions $F$.
	\end{definition}
	
	\begin{remark}
		We later show in Lemma~\ref{lem:SG-extension} that \eqref{eq:SG-inequality} implies analogous bounds
		with $L^p$-norms and for $F$ of polynomial growth.
		In our applications, it would in fact suffice to assume the spectral gap inequality~\eqref{eq:SG-inequality} just for polynomial cylindrical functions $F$, with the implicit constant possibly depending on the degree of the polynomial (see the proofs of \lem{lem:X^eps_Ceta_moment-bounds} and \theo{thm:convergence-of-mollifications}).
		The only change would be that Lemmas~\ref{lem:SG_2-->SG_p} and \ref{lem:SG-extension} would hold for $p=2^j\geq 2$ dyadic integers.
	\end{remark}
	
	\begin{lemma}\label{lem:SG_2-->SG_p}
		Suppose that the spectral gap inequality \eqref{eq:SG-inequality}
		holds in the sense of Definition~\ref{def:SG}.
		Then, for every
		$p\in[2,\infty)$ and bounded cylindrical function $F$,
		\begin{equ}\label{eq:SG_p-inequality}
			\E^{\frac1p}|F(\xi)|^p
			\lesssim
			|\E[F(\xi)]|
			+
			p\,
			\E^{\frac1p}
			\Big\|
			\frac{\delta F}{\delta\xi}
			\Big\|^p \;,
		\end{equ}
		where the implicit constant depends only on the constant in \eqref{eq:SG-inequality}.
	\end{lemma}
	
	\begin{remark}
		We do not use the exact linear dependence on $p$ in \eqref{eq:SG_p-inequality}, but we record it for completeness.
	\end{remark}
	
	\begin{proof}
		Denote
		\[
		A_p
		=
		\E^{\frac1p}|F(\xi)|^p\;,
		\qquad
		B_p
		=
		\E^{\frac1p}
		\Big\|
		\frac{\delta F}{\delta\xi}
		\Big\|^p\;.
		\]
		Let $C\ge1$ be such that \eqref{eq:SG-inequality} holds with
		implicit constant $C$.
		
		We first claim that, for every $p\geq 2$,
		\begin{equ}\label{eq:SG-recursion}
			A_p
			\leq
			(2C)^{2/p} A_{p/2}
			+
			Cp B_p\;.
		\end{equ}
		For $\epsilon>0$, we apply \eqref{eq:SG-inequality} to
		\[
		G_\epsilon(\xi)
		\eqdef
		\big(F(\xi)^2+\epsilon\big)^{p/4}\;.
		\]
		By the chain rule,
		\[
		\frac{\delta G_\epsilon}{\delta\xi}
		=
		\frac p2
		F(\xi)
		\big(F(\xi)^2+\epsilon\big)^{\frac p4-1}
		\frac{\delta F}{\delta\xi}\;.
		\]
		Since
		\[
		|F|
		\big(F^2+\epsilon\big)^{\frac p4-1}
		\leq
		\big(F^2+\epsilon\big)^{\frac{p-2}{4}}\;,
		\]
		letting $\epsilon\downarrow0$ in the resulting spectral gap
		estimate gives
		\begin{equ}\label{eq:SG-recursion-pre}
			A_p^{p/2}
			\leq
			C A_{p/2}^{p/2}
			+
			\frac{Cp}{2}
			\Big(
			\E\Big[
			|F(\xi)|^{p-2}
			\Big\|
			\frac{\delta F}{\delta\xi}
			\Big\|^2
			\Big]
			\Big)^{1/2}\;.
		\end{equ}
		By H\"older's inequality,
		\[
		\Big(
		\E\Big[
		|F(\xi)|^{p-2}
		\Big\|
		\frac{\delta F}{\delta\xi}
		\Big\|^2
		\Big]
		\Big)^{1/2}
		\leq
		A_p^{\frac p2-1} B_p\;,
		\]
		and hence
		\begin{equ}\label{eq:SG-recursion-pre2}
			A_p^{p/2}
			\le
			C A_{p/2}^{p/2}
			+
			\frac{Cp}{2}A_p^{\frac p2-1}B_p\;.
		\end{equ}
		
		If $A_p\leq Cp B_p$, then \eqref{eq:SG-recursion} is immediate.
		Otherwise $A_p>CpB_p$, so
		\[
		\frac{Cp}{2}A_p^{\frac p2-1}B_p
		\leq
		\frac12 A_p^{p/2}\;.
		\]
		It follows from \eqref{eq:SG-recursion-pre2} that $
		A_p^{p/2}
		\leq
		2C A_{p/2}^{p/2}$,
		and therefore
		$A_p
		\leq
		(2C)^{2/p}A_{p/2}$,
		proving \eqref{eq:SG-recursion}.
		
		We now iterate \eqref{eq:SG-recursion}.
		Denote $p_j=\frac{p}{2^j}$ for $j\geq 0$ and let $m\in\N$ be such that
		$p_m \in[1,2)$.
		Using $B_{p_j}\le B_p$, we obtain
		\begin{equation}\label{eq:SG_iteration}
			A_p
			\leq
			(2C)^{\sum_{j=0}^{m-1}2/p_j} A_{p_m}
			+
			Cp B_p
			\sum_{j=0}^{m-1}
			2^{-j}
			(2C)^{\sum_{\ell=0}^{j-1}2/p_\ell}\;.
		\end{equation}
		Since
		\[
		\sum_{j=0}^{m-1}\frac{2}{p_j}
		=
		\frac{2}{p}\sum_{j=0}^{m-1}2^j
		=
		\frac{2(2^m-1)}{p}
		<2\;,
		\]
		all the powers of $2C$ appearing in \eqref{eq:SG_iteration}
		are bounded by $(2C)^2$, uniformly in $p$.
		Moreover $\sum_{j=0}^{m-1}2^{-j}\leq 2$,
		so we obtain
		\begin{equ}\label{eq:SG_iteration_conclusion}
			A_p
			\lesssim
			A_{p_m}
			+
			p B_p\;,
		\end{equ}
		with an implicit constant depending only on $C$.
		Since $p_m<2$, monotonicity of the $L^r(\P)$-norms together with \eqref{eq:SG-inequality} applied to $F$ yields
		\[
		A_{p_m}\leq A_2 \leq
		C\big(|\E[F(\xi)]|+B_2\big)
		\leq
		C\big(|\E[F(\xi)]|+B_p\big)\;.
		\]
		Combining this with \eqref{eq:SG_iteration_conclusion} yields
		$A_p
		\lesssim
		|\E[F(\xi)]|
		+
		p B_p$,
		which is \eqref{eq:SG_p-inequality}.
	\end{proof}
	
	\begin{lemma}\label{lem:SG-extension}
		Suppose that $\xi$ satisfies the spectral gap inequality
		\eqref{eq:SG-inequality} in the sense of Definition~\ref{def:SG}.
		Then the following statements hold.
		
		\begin{enumerate}[label=(\roman*)]
			
			\item\label{item:linear-moments}
			For every $\phi\in\CD(\T^n;E)$,
			$\xi(\phi)\in L^p(\P)$ for all $p\in[1,\infty)$.
			If, moreover, the norm $\|\cdot\|$ in \eqref{eq:SG-inequality} is continuous on
			$\CD(\T^n;E)$, then $\scal{\E\xi,\phi}
			\eqdef \E[\xi(\phi)]$
			defines an element $\E\xi\in\CD'(\T^n;E)$.
			
			\item\label{item:SG-extension-general}
			Let $p\in [2,\infty)$ and $F$ be a cylindrical function such that
			$F(\xi)\in L^1(\P)$
			and
			$\|\frac{\delta F}{\delta\xi}\|\in L^p(\P)$.
			Then $F(\xi)\in L^p(\P)$ and
			\begin{equ}\label{eq:SG_p-extension}
				\|F(\xi)\|_{L^p}
				\lesssim
				|\E F(\xi)|
				+
				p\,
				\Big\|
				\Big\|\frac{\delta F}{\delta\xi}\Big\|
				\Big\|_{L^p}\;,
				\tag{$\mathrm{SG}_p$}
			\end{equ}
			where the implicit constant depends only on the constant in \eqref{eq:SG-inequality}.
			
			\item\label{item:polynomial-cylindrical}
			\eqref{eq:SG_p-extension} holds for all $p\in [2,\infty)$ and cylindrical $F(\xi)=f(\xi(\phi_1),\ldots,\xi(\phi_m))$ for which
			$f$ and its first derivative have at most polynomial growth.
			For every such $F$, both $F(\xi)$ and $\|\frac{\delta F}{\delta\xi}\|$ are in $L^p(\P)$ for all $p\in [1,\infty)$.
		\end{enumerate}
	\end{lemma}
	
	\begin{proof}
		\ref{item:linear-moments} It suffices to consider $p\geq 2$.
		Let $Z = \xi(\phi)$.
		Choose $M<\infty$ such that
		\[
		\P(A)\geq 1/2\;,
		\]
		where we define the event $A\eqdef\{|Z|\leq M\}$.
		For $R\geq M$, let $\rho_R\colon\R\to\R$ be a smooth bounded function
		such that
		\[
		\rho_R(x)=x \quad\text{for } |x|\leq R\;,
		\qquad
		|\rho_R'(x)|\leq1\;,
		\qquad
		|\rho_R(x)|\leq|x|\;,
		\]
		and set
		$Z_R = \rho_R(Z)$.
		Then $Z_R$ is a bounded cylindrical function with a bounded derivative.
		Applying
		\eqref{eq:SG-inequality} to $Z_R-\E Z_R$ yields
		\begin{equ}\label{eq:ZR-variance}
			\|Z_R-\E Z_R\|_{L^2}
			\lesssim
			\|\phi\|
		\end{equ}
		because $\frac{\delta Z_R}{\delta\xi}
		=
		\rho_R'(Z)\phi$
		and $|\rho_R'|\leq1$.
		
		We next control $\E Z_R$. On the event $A$,
		we have $Z_R=Z$, and hence $|Z_R|\leq M$. Therefore, on $A$,
		\[
		|Z_R-\E Z_R|
		\geq
		\big(|\E Z_R|-M\big)_+\;,
		\]
		where $x_+ = \max\{x,0\}$.
		Since $\P(A)\geq1/2$,
		\[
		\|Z_R-\E Z_R\|_{L^2}^2
		\geq
		\frac12\big(|\E Z_R|-M\big)_+^2\;.
		\]
		Together with \eqref{eq:ZR-variance}, we obtain
		\[
		|\E Z_R|
		\lesssim
		M+\|\phi\|
		\]
		uniformly in $R$.
		Applying \eqref{eq:SG_p-inequality} to the bounded cylindrical
		function $Z_R$, for $p\geq2$, therefore gives
		\[
		\|Z_R\|_{L^p}
		\lesssim
		M+p\|\phi\|\;,
		\]
		uniformly in $R$. Since $Z_R\to Z$ almost surely as $R\to\infty$,
		Fatou's lemma yields
		\[
		\|Z\|_{L^p}
		\lesssim
		M+p\|\phi\|\;,
		\]
		which proves $\xi(\phi)\in L^p(\P)$.
		Assume now that $\|\cdot\|$ is continuous on $\CD(\T^n;E)$.
		Set
		\[
		m(\phi) = \E[\xi(\phi)]\;,
		\]
		which is a linear functional on $\CD(\T^n;E)$.
		By dominated convergence, we can take the $R\to\infty$ limit in \eqref{eq:ZR-variance}
		and obtain
		\begin{equ}\label{eq:centered-linear-SG}
			\|\xi(\phi)-m(\phi)\|_{L^2}
			\lesssim \|\phi\|\;.
		\end{equ}
		We claim that $m$ is continuous on $\CD$.
		Indeed, suppose
		$\phi_j\to 0$ in $\CD$.
		Then, by \eqref{eq:centered-linear-SG},
		\[
		\|\xi(\phi_j)-m(\phi_j)\|_{L^2}
		\lesssim\|\phi_j\|\to0\;,
		\]
		where we used the continuity of $\|\cdot\|$ on $\CD$.
		It follows that $\xi(\phi_j)-m(\phi_j) \to 0$ in probability.
		Moreover, $\xi(\phi_j)\to0$ almost surely since $\xi\in\CD'$ almost surely.
		Therefore the deterministic sequence $m(\phi_j)$ converges to $0$ in probability, hence in $\R$.
		It follows that $m$ is continuous on $\CD$, proving \ref{item:linear-moments}.
		
		\ref{item:SG-extension-general} Let $\rho_R$ be as
		above and denote $F_R\eqdef\rho_R(F)$.
		Then $F_R$ is a bounded cylindrical function and
		\[
		\frac{\delta F_R}{\delta\xi}
		=
		\rho_R'(F)\frac{\delta F}{\delta\xi}\;.
		\]
		Consequently, by \eqref{eq:SG_p-inequality},
		\[
		\|F_R\|_{L^p}
		\lesssim
		|\E F_R|
		+
		p\,
		\Big\|
		\Big\|\frac{\delta F}{\delta\xi}\Big\|
		\Big\|_{L^p}\,.
		\]
		Since $|\rho_R(F)|\leq|F|$ and $F\in L^1$, dominated convergence implies $\E F_R\to \E F$.
		Fatou's lemma therefore yields both $F\in L^p$ and
		\eqref{eq:SG_p-extension}.
		
		\ref{item:polynomial-cylindrical} Suppose
		\[
		F(\xi)=
		f\big(\xi(\phi_1),\ldots,\xi(\phi_m)\big)
		\]
		and that $f$ and its first derivatives have polynomial growth.
		By \ref{item:linear-moments}, each $\xi(\phi_i)$ has moments of every
		finite order, and hence so does $F(\xi)$. Moreover,
		\[
		\Big\|\frac{\delta F}{\delta\xi}\Big\|
		\leq
		\sum_{i=1}^m
		|\partial_i f(\xi(\phi_1),\ldots,\xi(\phi_m))|
		\,\|\phi_i\|\;,
		\]
		so the functional derivative also has moments of every finite order.
		Thus \ref{item:SG-extension-general} applies for every $p\geq2$,
		proving \ref{item:polynomial-cylindrical}.
	\end{proof}
	
	\begin{remark}\label{rem:centered}
		Under Assumption~\ref{ass:field},
		Lemma~\ref{lem:SG-extension}\ref{item:linear-moments}
		implies that $\E X\in\CD'(\T^n;E)$ and,
		by stationarity, $\E X$ is translation invariant and hence constant.
	\end{remark}
	
	\begin{lemma}\label{lem:X^eps_Ceta_moment-bounds}
		Let $q > 1$ with H\"older conjugate $q'$.
		Let $X$ be an $E$-valued random field on $\T^n$ satisfying the spectral gap inequality \eqref{eq:SG-inequality} with $L^{q'}$-norm in the sense of Definition~\ref{def:SG}.
		Suppose $\E X$ (well-defined by Lemma~\ref{lem:SG-extension}\ref{item:linear-moments}) is a constant.
		Let $X^{\e}$ be a mollification at scale $\e>0$ of $X$ defined in \theo{thm:main}
		and denote $X^0 = X$.
		
		Then, for any $\nu\in[0,1)$, 
		$\eta<-n/q-\nu$, 
		and $p\ge 1$, we have
		\begin{equ}
			\E\Big|\sup_{\e\in[0,1]}|X^\e|_{\CC^\eta}\Big|^p <\infty\;,
			\qquad
			\E\Big|\sup_{0\leq\bar\e<\e\leq 1}\frac{|X^\e-X^{\bar\e}|_{\CC^\eta}}{|\e-\bar\e|^\nu}\Big|^p <\infty\;.
		\end{equ}
	\end{lemma}

	\begin{proof}
		Denote $m=\E X$. Because $\chi^\e * m =m$ since $m$ is constant, it suffices to prove the bounds for $X-\E X$, which satisfies the same spectral gap inequality.
		We prove only the second bound; for the first bound, we can show $\E|X|_{\CC^\eta}^p<\infty$ in a similar and simpler manner, and then the first bound follows from the second one.
		
		Fix $\nu\in[0,1)$ and $\eta<-n/q-\nu$.
		Choose $\kappa\in[0,1]$ such that
		$\nu<\kappa <-n/q-\eta$
		and set $\gamma=-n/q-\kappa>\eta$.
		By a standard Kolmogorov argument
		(e.g. \cite[Theorem~2.7]{ChandraWeber17} for estimates in $\CC^\eta$ followed by 
		\cite[Theorem~A.10]{FV10} applied to the $1$-dimensional $\e$-parameter), it suffices to prove that, for all $p\geq 1$, uniformly in $\e,\bar\e\in[0,1]$, $\phi\in\CB^r$, $z\in\T^n$, $\lambda\in(0,1]$,
		\begin{equation}\label{eq:X_diff_moments}
			\lambda^{-\gamma}\, \E^{\frac1p} \big|\scal{X^\e-X^{\bar\e}, \phi_z^\lambda}\big|^p 
			\lesssim
			|\e-\bar\e|^{\kappa}\;.
		\end{equation}
		Here $r=-\floor{\eta}+1$, $\CB^r$, and $\phi_z^\lambda$ are as in Section~\ref{sec:notation}.
		
		By \lem{lem:SG-extension}, $X$ satisfies \eqref{eq:SG_p-extension} for $L^{q'}$-norm for any $p\geq 2$. Since $X$ is centred, duality between $L^{q'}$ and $L^q$, implies for all $p\geq 2$
		\begin{equs}
			\lambda^{-\gamma}\, \E^{\frac1p} \big|\scal{X^\e-X^{\bar\e}, \phi_z^\lambda}\big|^p 
			&\lesssim
			\lambda^{-\gamma}\!
			\sup_{|h|_{L^{{q}}}= 1}\!\! \big|\scal{h^\e-h^{\bar\e}, \phi_z^\lambda}\big| 
			\le\! \sup_{|h|_{L^{{q}}}= 1}\!\!
			|h^\e-h^{\bar\e}|_{\CC^{\gamma}} 
			\\&  
			\lesssim \sup_{|h|_{L^{{q}}}= 1} \!
			|h|_{\CC^{\gamma+\kappa}}
			|\e-\bar\e|^{\kappa}\;,
		\end{equs}
		where $h^\epsilon \eqdef \chi^\epsilon * h$ for $\epsilon \in \{\e, \bar\e\}$, and the final bound follows from a standard convolution estimate (see \lem{lem:convolution-estimate}).
		
		Since $\gamma+\kappa=-n/q<0$, the embedding
		$L^q=L^{-n/(\gamma+\kappa)}\hookrightarrow \CC^{\gamma+\kappa}$
		from Lemma~\ref{lem:L-HB-embedding} yields
		\begin{equ}
			\sup_{|h|_{L^{{q}}}= 1} \!
			|h|_{\CC^{\gamma+\kappa}}
			|\e-\bar\e|^{\kappa}
			\lesssim
			|\e-\bar\e|^{\kappa}\;,
		\end{equ}
		from which \eqref{eq:X_diff_moments} now follows.
	\end{proof}
	
	\begin{lemma}\label{lem:oddness}
		Consider $\tau \in \CT$. The following statements hold.
		\begin{enumerate}[label=(\roman*)]
			\item\label{item:odd}
			If $k$ is odd, then $\der{\tau} + \noise{\tau}$ is odd, where $\der{\tau}$ denotes the number of edges of $\tau$ labelled with $I'$.
			\item\label{item:even}
			If $k$ is even, then $\noise{\tau}$ is odd.
		\end{enumerate}
	\end{lemma}
	\begin{proof}
		\noindent\emph{\ref{item:odd} $k$ odd.}
		We argue by induction on $\noise{\tau}$. The base case $\noise{\tau}=1$, corresponding to $\tau=\<Xi>$, is immediate.
		
		Assume the claim holds for all $\sigma \in \CT^m_{\<Xi>}$ with $m \ge 1$, and let $\tau \in T_{m+1}$. Then
		\begin{equs}[eq:inductive-parity]
			\der{\tau}+\noise{\tau}
			=
			\begin{cases}
				\displaystyle
				1+\sum_{i=1}^{k+1}
				\bigl(\der{\tau_i}+\noise{\tau_i}\bigr)\,,
				&
				\displaystyle
				\tau=I'(\tau_{k+1})\prod_{i=1}^{k}\!I(\tau_i)\;, \;\; \{\tau_i\}_{i=1}^{k+1}\subset\CT^{m}_{\<Xi>} \;,
				\\
				\displaystyle
				\sum_{i=1}^{2k+1}
				\bigl(\der{\tau_i}+\noise{\tau_i}\bigr)\,,
				&
				\displaystyle
				\tau= \prod_{i=1}^{2k+1}\!I(\tau_i)\;,\;\; \{\tau_i\}_{i=1}^{2k+1}\subset\CT^{m}_{\<Xi>} \;.
			\end{cases}
		\end{equs}
		Therefore, the oddness of $\der{\tau}+\noise{\tau}$ in both cases follows immediately from the inductive hypothesis and the fact that $k$ is odd.
		
		\smallskip
		
		\noindent\emph{\ref{item:even} $k$ even.}
		This follows by a similar induction on $\noise{\tau}$, using
		\begin{equ}
			\noise{\tau}=
			\begin{cases}
				\displaystyle
				\sum_{i=1}^{k+1} \noise{\tau_i}\,, 
				& \; 
				\displaystyle
				\tau=I'(\tau_{k+1})\prod_{i=1}^{k}\!I(\tau_i)\;, \;\; \{\tau_i\}_{i=1}^{k+1}\subset\CT^{m}_{\<Xi>}\,,
				\\[1em]
				\displaystyle
				\sum_{i=1}^{2k+1} \noise{\tau_i}\,, 
				& \; 
				\displaystyle
				\tau= \prod_{i=1}^{2k+1}\!I(\tau_i)\;, \;\; \{\tau_i\}_{i=1}^{2k+1}\subset\CT^{m}_{\<Xi>} 
			\end{cases}
		\end{equ}
		in place of \eqref{eq:inductive-parity}, together with the fact that $k$ is even.
	\end{proof}
	For later reference, we recall (from \defi{def:trees}) that for a rooted tree $\tau$, we write $V_{\tau}$ for its set of vertices. 
	We denote by $L_{\tau}$ the set of leaves of $\tau$ and let $I_{\tau} \eqdef V_{\tau} \setminus L_{\tau}$, the elements of which are referred to as inner vertices.  
	Furthermore, we define 
	\[
	V^*_{\tau} \eqdef V_{\tau} \setminus \{\rho\}\;, \qquad
	L^*_{\tau} \eqdef L_{\tau} \setminus \{\rho\}\;, \qquad
	I^*_{\tau} \eqdef I_{\tau}  \setminus \{\rho\}\;,
	\]  
	where $\rho$ denotes the root of $\tau$.
	We denote henceforth the periodic heat kernel by
	\[
	G_t(x) = \sum_{z\in \Z^n} \frac{1}{(4\pi t)^{n/2}} \exp\Big(-\frac{|x+z|^2}{4t}\Big)
	\]
	for $t>0$ and $x\in\T^n$, and use the convention $G_t(x) = 0$ for $t\leq 0$.
	
	\begin{lemma}\label{lem:X^tau,e_mean-zero}
		Let $X$ be an $E$-valued random field satisfying
		\assu{ass:field}. Denote by $X^\e$ the mollification of $X$ at scale $\e\in(0,1)$,
		as defined in \theo{thm:main}, and let $\tau\in\CT\setminus\{\<Xi>\}$.
		Then $X^{\tau,\e}_t (x)$ is $\P$-integrable and $\E\big[X^{\tau,\e}_t (x)\big]=0$ for any $x \in \T^n$ and $t\in (0,1)$,
		where $X^{\tau,\e}\eqdef (X^\e)^\tau$.
		
		More generally,
		the result still holds if, at each leaf $v\in L_\tau$, the field $X^\e$ is replaced by a linear combination $f_v = \sum_{i=1}^{m_v} \lambda_i X^{\e_v^i}$, where $\lambda_i\in\R$, $m_v\in \N$, and $\{\e_v^{i}\}_{i\in[m_v],\, v\in L_\tau}$ is a collection of positive numbers.
	\end{lemma}
	
	\begin{proof}
		We only detail the proof of the first statement, as the more general case follows in exactly the same way.
		
		First, note that $X^{\tau,\e}_t(x)$ can be written as a linear combination of finitely many terms of the form
		\begin{equation}\label{eq:Y^tau}
			X^{\tau,h,j}_t(x) = \int_{(0,1)^{I_\tau^*}}\mrd t \int_{(\T^n)^{V^{*}_{\tau}}} \mrd x
			\prod_{v\in L_{\tau}} X^{h_v}(x_{v})
			\prod_{u\in V^{*}_{\tau}} G^{(j_{\hat{u},u})}_{t_{\hat{u}}-t_{u}}(x_{\hat{u}}-x_{u})\;,
		\end{equation}
		where $h \in (E^{*})^{L_\tau}$ and $X^{h_v}(\cdot) \eqdef h_v(X^\e(\cdot))$ and we set $t_u = 0$ for every leaf $u \in L_\tau^*$. We denote $(t_\rho,x_\rho) = (t,x)$ for $\rho$ the root of $\tau$. Moreover, $j\in (\N^n_0)^{E_\tau}$ satisfies $|j_{\hat u,u}| = 1$ if $(\hat{u},u)$ carries the label $I'$, and $j_{\hat u,u} = 0$ otherwise.
		
		It therefore suffices to show that each $X^{\tau,h,j}_t(x)$ is $\P$-integrable and has zero expectation.
		First, for fixed $\e>0$, by Lemma~\ref{lem:SG-extension} and stationarity, $X^\e(x) = X(\chi^\e(x-\cdot))$ has moments of all orders uniformly in $x$,
		thus so does $\prod_{v\in L_{\tau}} X^{h_v}(x_{v})$ uniformly in $x_v$.
		Furthermore, $|G_t|_{L^1(\T^n)}=1$ and $|G^{(j)}_t|_{L^1(\T^n)}\lesssim t^{-1/2}$ for $t>0$ and $|j|=1$, so the time singularities in \eqref{eq:Y^tau} are integrable,
		and we conclude by
		Tonelli's theorem that $X^{\tau,h,j}_t(x)$ is $\P$-integrable.
		
		Next, by stationarity, it is enough to prove that $\E[X^{\tau,h,j}_t(0)] = 0$. Indeed, performing the change of variables $x_v \mapsto x_v + x$ for all $v \in V_\tau^*$ and using translation invariance of the heat kernel, all kernels remain unchanged except those of the form $G^{(j)}_{t_\rho-\cdot}(x_\rho-\cdot)$, which become $G^{(j)}_{t_\rho-\cdot}(-\cdot)$. 
		Hence, by stationarity of $X$, $\E[X^{\tau,h,j}_t(x)] = \E[X^{\tau,h,j}_t(0)]$.
		
		If $X^{\tau,h,j} \equiv 0$ then we are done.
		Assume henceforth that $X^{\tau,h,j} \not\equiv 0$, which in particular implies
		$F_{\crit}\not\equiv0$.
		Let $\bipar$ be as in Assumption~\ref{ass:field}.
		We claim that
		\begin{equation}\label{eq:parity}
			\noise{\tau} + \bipar\der{\tau} \; \text{ is odd.}
		\end{equation}
		Indeed, 
		if $Q_k \not\equiv 0$ and $k$ is odd, then $\bipar=1$ and Lemma~\ref{lem:oddness}\ref{item:odd} implies that $\noise{\tau} + \bipar\der{\tau}$ is odd.
		If $Q_k \not\equiv 0$ and $k$ is even, then $\bipar=0$ and Lemma~\ref{lem:oddness}\ref{item:even} implies that $\noise{\tau}$ is odd.
		If $Q_k \equiv 0$, then $X^{\tau,h,j} \not\equiv 0$ implies $\der{\tau}=0$ and $\tau$ is $(2k+1)$-ary,
		therefore $\noise{\tau} + \bipar\der{\tau} = \noise{\tau}$ is odd.
		This proves \eqref{eq:parity}.
		
		Since $\chi$ is even, mollification commutes with reflection, and hence, by \eqref{eq:compatible-symmetry},
		\begin{equation}\label{eq:reflection-commutes}
			\CR^{\bipar} X^\e
			=
			(\CR^{\bipar} X)^\e
			\eqdist
			-X^\e\;.
		\end{equation}
		Consider \eqref{eq:Y^tau} with $x=0$ and perform the change of variables
		$x_v\mapsto (-1)^{\bipar} x_v$ for all
		$v\in V_\tau^*$.
		Since $G_t$ is even in space and every $G_t^{(j)}$ with $|j|=1$ is
		odd, the product of heat kernels is multiplied by $(-1)^{\bipar\der{\tau}}$.
		On the other hand, by \eqref{eq:reflection-commutes},
		\[
		\prod_{v\in L_\tau}
		X^{h_v}\big((-1)^{\bipar} x_v\big)
		\eqdist
		(-1)^{\noise{\tau}}
		\prod_{v\in L_\tau}X^{h_v}(x_v)\;.
		\]
		We thus obtain
		\[
		\E X^{\tau,h,j}_t(0)
		=
		(-1)^{\noise{\tau}+\bipar\der{\tau}}
		\E X^{\tau,h,j}_t(0)\;.
		\]
		By \eqref{eq:parity}, $\noise{\tau}+\bipar\der{\tau}$
		is odd, hence
		$\E X^{\tau,h,j}_t(0)=0$.
	\end{proof}

	\begin{lemma}
		\label{lem:heat-flow-estimates}
		Let $\CP_t = \mre^{t\Delta}$ denote the heat semigroup on $\T^n$, and let $t \in (0,1]$.
		\begin{enumerate}[label=(\roman*)]
			\item
			\label{item:Holder}
			Let $a\in\N_0^n$ and $b,\bar b\in\R$ with
			$b\le \bar b$. Then, for every $f\in\CC^b(\T^n)$,
			\begin{equ}
				|D^{a} \CP_t f|_{\CC^{\bar b}}
				\;\lesssim\;
				t^{\frac{b- \bar b-|a|}{2}}
				|f|_{\CC^{b}}\;.
			\end{equ}
			
			\item
			\label{item:Leb}
			Let $a\in\N^n_{0}$ and $1 \le r \le \bar r \le \infty$. Then, for all $f\in L^r(\T^n)$,
			\begin{equ}
				|D^a \CP_t f|_{L^{\bar r}}
				\;\lesssim\;
				t^{-\frac{|a|}{2}-\frac n2(\frac1r-\frac{1}{\bar r})}
				|f|_{L^r}\;.
			\end{equ}
			
			\item
			\label{item:Hol_diff}
			Let $a\in\N^n_{0}$, $\,b, \bar b\in\R$ with $\,b\le \bar b$, and $f\in \CC^b(\T^n)$. Then, for all $s\in(0,t]$ and $\kappa\in[0,1]$,
			\begin{equ}
				|D^a(\CP_t-\CP_s)f|_{\CC^{\bar b}}
				\;\lesssim\;
				|t-s|^\kappa\,
				s^{\frac{b-\bar b-|a|}{2}-\kappa}
				|f|_{\CC^b}\;.
			\end{equ}

			\item
			\label{item:Leb_diff}
			Let $a\in\N^n_{0}$, $\,1\le r\le \bar r\le\infty$,  and $f\in L^r(\T^n)$. Then, for all $s\in(0,t]$ and $\kappa\in[0,1]$,
			\begin{equ}
				|D^a (\CP_t-\CP_s)f|_{L^{\bar r}}
				\;\lesssim\;
				|t-s|^\kappa\,
				s^{-\frac{|a|}{2} -\frac n2(\frac1r-\frac{1}{\bar r})-\kappa}
				|f|_{L^r}\;.
			\end{equ}
		\end{enumerate}
		
		The implicit constants depend only on $n,b,\bar b,r,\bar r,a$.
	\end{lemma}
	
	\begin{proof}
		\ref{item:Holder} follows immediately from
		\eqref{eq:heat_flow_estimates} and the estimate $|D^a g|_{\CC^{\bar b}} \lesssim |g|_{\CC^{\bar b+|a|}}$.
		
		To prove \ref{item:Leb},
		recall that $G_t(x)$ denotes the periodic heat kernel. For every
		$a\in\N_0^n$ and $q\in[1,\infty]$, we have uniformly in $t\in(0,1]$
		\begin{equation}\label{eq:heat-kernel-Lq}
			|D^aG_t|_{L^q}
			\lesssim
			t^{-\frac{|a|}{2}-\frac n2(1-\frac1q)}\;,
		\end{equation}
		which follows directly from the corresponding Gaussian
		derivative estimate on $\R^n$ and the representation of $G_t$ as
		the periodisation of the heat kernel on $\R^n$.
		Now choose $q\in[1,\infty]$ so that
		$1+\frac1{\bar r}=\frac1r+\frac1q$.
		Young's inequality and \eqref{eq:heat-kernel-Lq} then imply
		\[
		|D^a\CP_t f|_{L^{\bar r}}
		\leq
		|D^aG_t|_{L^q}|f|_{L^r}
		\lesssim
		t^{-\frac{|a|}{2}
			-\frac n2(\frac1r-\frac1{\bar r})}
		|f|_{L^r}\,.
		\]
		For \ref{item:Hol_diff}.
		we have $|D^a(\CP_t-\CP_s)f|_{\CC^{\bar b}}\lesssim s^{\frac{b-\bar b -|a|}{2}}|f|_{\CC^b}$ by \ref{item:Holder}.
		On the other hand, writing $\CP_t-\CP_s = \int_s^t\Delta\CP_r \mrd r$ and using $|D^a \Delta \CP_r f|_{\CC^{\bar b}}\lesssim r^{\frac{b-\bar b-|a|}{2}-1}|f|_{\CC^{b}}$ by \ref{item:Holder}, we obtain
		\[
		|D^a(\CP_t-\CP_s)f|_{\CC^{\bar b}} \lesssim \int_s^t r^{\frac{b-\bar b-|a|}{2}-1} |f|_{\CC^b} \mrd r \lesssim |t-s| s^{\frac{b-\bar b-|a|}{2}-1}|f|_{\CC^b}\;.
		\]
		Interpolating the two bounds proves \ref{item:Hol_diff}.
		\ref{item:Leb_diff} follows in exactly the same manner from \ref{item:Leb}.
	\end{proof}
	
	\begin{proof}[of \theo{thm:convergence-of-mollifications}]
		By Lemma~\ref{lem:X^eps_Ceta_moment-bounds} and Remark~\ref{rem:centered}, we have
		\begin{equ}[eq:X_eps_diff]
			\||X^\e-X^{\bar\e}|_{\CC^{\eta}}\|_{L^p} \lesssim |\e-\bar\e|^{\ubar{\kappa}}\;,
		\end{equ} 
		for any $\ubar{\kappa}\in[0,1)$, $\eta<\alpha-\ubar{\kappa}$, and all $p\ge 1$, and uniformly in $\e,\bar\e\in[0,1]$.
		Lemma~\ref{lem:X^eps_Ceta_moment-bounds} also implies that \eqref{eq:X_eps_diff} holds with $\ubar{\kappa}=0$ and $X^\e-X^{\bar\e}$ on the left replaced by either $X^\e$ or $X^{\bar\e}$.
		Below, when invoking \eqref{eq:X_eps_diff} with $\ubar{\kappa}=0$, the estimate is understood in this no-difference sense.
		
		By a standard Kolmogorov-type argument (see e.g. the proof of \cite[Corollary~3.3]{CM25}), it suffices to show that there exist $\bar\kappa>0$ such that for any $\kappa\in[0,\bar\kappa]$,
		$\tau\in\CT^N$, $\beta \in [\alpha,0]$, $\delta>-|\tau|/2+\beta/2$, 
		and all $p\geq1$, and uniformly in $t,s\in(0,1)$, $\e,\bar\e\in(0,1)$, $\lambda\in(0,1]$, $z\in\T^n$,
		and $\phi\in\CB^r$ with $r=-\floor{\beta}+1$ (see Section~\ref{sec:notation}),
		\begin{align}
			\label{eq:Z_eps_diff.}
			\| \scal{t^{\delta+{\kappa}/2} X^{\tau,\e}_t-t^{\delta+{\kappa}/2} X^{\tau,\bar\e}_t , \phi^\lambda_z} \|_{L^p}
			&\lesssim |\e-\bar\e|^{{\kappa}} \lambda^{\beta}\;,
			\\
			\label{eq:Z_time_diff.}
			\|  \scal{t^{\delta+\kappa/2} X_t^{\tau,\e} -s^{\delta+\kappa/2} X_s^{\tau,\e},\phi^\lambda_z} \|_{L^p}
			&\lesssim |t-s|^{\kappa/2} \lambda^{\beta}\;.
		\end{align}
		
		\medskip
		
		\textit{Proof of \eqref{eq:Z_eps_diff.}.} We employ a telescoping argument combined with the triangle inequality. To prove \eqref{eq:Z_eps_diff.}, it is therefore sufficient to establish that
		\begin{equ}[eq:Y_bound]
			\mathbb{E}\big|\langle Y^\tau_{t}, \phi_z^\lambda\rangle\big|^p \;\lesssim\;
			t^{-p(\delta+\kappa/2)} \lambda^{p\beta}\,|\varepsilon-\bar\varepsilon|^{p\kappa}\,.
		\end{equ}
		Here, for any $(t,x)\in (0,1)\times\mathbb{T}^n$, the field $Y^\tau_t(x)$ is defined analogously to $X^\tau_t(x)$, with the only difference that in the tree $\tau$ one leaf corresponds to $X^\varepsilon - X^{\bar\varepsilon}$, while all remaining leaves correspond to either $X^\varepsilon$ or $X^{\bar\varepsilon}$.
		
		To apply the spectral gap inequality, we approximate $X$ by a smooth finite-dimensional Fourier truncation
		$X^{(R)} = S_R X$
		(e.g. a Littlewood--Paley projector) for $R\geq 1$.
		Here $S_R \colon \CD'\to\CC^\infty$
		is a linear operator taking values in the linear span of $(e_j)_{|j|\le R}$ with $e_j(x) = \mre^{2\pi\mri j\cdot x}$ for $j\in\mathbb{Z}^n$ and $x\in\mathbb{T}^n$,
		and such that $S_R f\to f$ in $\CD'(\T^n)$ as $R\to\infty$ for every $f\in\CD'$, and $S_R$ is self-adjoint, commutes with translations, reflections, and is uniformly bounded in $R\geq 1$ as an operator on $\CC^\gamma$ for any $\gamma\in\R$ and on $L^r(\T^n)$ for any $r\in[1,\infty]$.
		It is easy to show that $X^{(R)}$ satisfies Assumption~\ref{ass:field} with a constant in \eqref{eq:SG-inequality} that is uniform in $R$.
		
		We then define the corresponding approximations $Y^{\tau,(R)}_t$ by replacing each occurrence of $X$ in the recursive construction of $Y^\tau$ with $X^{(R)}$, and set
		\[
		F_R \eqdef \langle Y^{\tau,(R)}_t, \phi^\lambda_z\rangle\;.
		\]
		By construction, $F_R$ is a cylindrical functional of $X^{(R)}$, so that \eqref{eq:SG_p-extension} applies by Lemma~\ref{lem:SG-extension}.
		Below we will obtain moment bounds for $F_R$ that are uniform in $R$ since \eqref{eq:SG-inequality} holds for $X^{(R)}$ uniformly in $R$.
		Since $F_R\to \scal{Y^\tau_t,\phi^\lambda_z}$ almost surely as $R\to\infty$, Fatou's lemma implies the corresponding moment bounds for $\scal{Y^\tau_t,\phi^\lambda_z}$.
		For notational simplicity, and with a slight abuse of notation, we suppress the dependence on $R$ of all objects in what follows.
		
		\medskip
		
		We have, by \lem{lem:X^tau,e_mean-zero},
		that $\E \scal{Y^{\tau}_{t}, \phi_z^\lambda}=0$.
		Hence, using the spectral gap inequality from \assu{ass:field},
		\lem{lem:SG-extension}\ref{item:polynomial-cylindrical}, and duality between $L^{q_*}$ and $L^{q^*}$, we obtain for all $p\geq 2$
		\begin{equ}[eq:SG-for-Y^tau]
			\E \big|\scal{Y^{\tau}_{t}, \phi_z^\lambda}\big|^p\, \lesssim \,
			\E\! \sup_{|h|_{L^{q_*}}= 1}\! \max_{\tau^h} \big|\scal{Y^{\tau^h}_{t}, \phi_z^\lambda}\big|^p\;,
		\end{equ}
		where the $\max$ is over all trees $\tau^h$ isomorphic to $\tau$ with one of the leaves representing the element $h$ (more precisely $h^{(R)}$, which is also uniformly bounded in $L^{q_*}$ by our choice of Fourier truncation) rather than the corresponding field, and $Y^{\tau^h}$ is defined correspondingly.
		(Since the dependence on the choice of the leaf corresponding to $h$ is not relevant for the analysis, with a slight abuse of notation, we suppress it in the notation.)
		
		\smallskip
		Using the Besov embedding $L^{-n/\beta} \hookrightarrow \CC^{\beta}$ (see Lemma~\ref{lem:L-HB-embedding})
		and that
		\begin{equ}[eq:embed-ineq]
			\lambda^{-\beta} \big|\scal{Y^{\tau^h}_{t}, \phi_z^\lambda}\big| \le |Y^{\tau^h}_{t}|_{\CC^{\beta}}
			\lesssim  |Y^{\tau^h}_{t}|_{L^{-n/\beta}}\;,
		\end{equ}
		we conclude \eqref{eq:Y_bound} once we establish
		\begin{equ}[eq:inductive-hypothesis]
			\Big\| \sup_{|h|_{L^{q_*}}= 1} |Y^{\tau^h}_{t}|_{L^{-n/\beta}} \Big\|_{L^p} \lesssim
			t^{-\delta-\kappa/2}|\e-\bar\e|^{\kappa}\;,
		\end{equ}
		where the parameters are as specified earlier.
		Another useful consequence of \eqref{eq:SG-for-Y^tau}--\eqref{eq:inductive-hypothesis},
		combined with a Kolmogorov argument, 
		is
		\begin{equ}[eq:Kol-C_beta]
			\Big\| |Y^{\tau}_{t}|_{\CC^{\beta^-}} \Big\|_{L^p} \lesssim
			t^{-\delta-\kappa/2} \, |\e-\bar\e|^{\kappa}
		\end{equ}
		for every $\beta^-<\beta$. 
		Here and below, we adopt the following endpoint convention. When
		$\beta=0$, every occurrence of
		$- n/\beta$ and $L^{-n/\beta}$
		is interpreted, respectively, as
		$\infty$ and $L^{\infty}$.		
		At this endpoint, rather than using the embedding
		$L^{-n/\beta}\hookrightarrow\CC^\beta$, we use directly
		\[
		\big|\scal{f,\phi_z^\lambda}\big|
		\le
		|f|_{L^\infty}|\phi_z^\lambda|_{L^1}
		\lesssim
		|f|_{L^\infty}\,.
		\]
		The heat flow estimates used below remain valid for $q=\infty$.
		Accordingly, when $\beta=0$, the spatial Kolmogorov argument yields
		$\CC^{\beta^-}$-regularity for every $\beta^-<0$.
		
		At this point, we also note that the no-difference counterparts of
		\eqref{eq:SG-for-Y^tau}--\eqref{eq:Kol-C_beta} 
		hold with
		$\kappa=0$ on the right-hand side whenever the tree on the left-hand
		side contains no leaf corresponding to convolution with
		$\chi^\e-\chi^{\bar\e}$. 
		Since their proofs are identical, we omit the details and regard the
		no-difference counterparts as established. In the arguments below,
		whenever we apply any of
		\eqref{eq:SG-for-Y^tau}--\eqref{eq:Kol-C_beta}  
		with
		$\kappa=0$ to such a tree, the estimate is understood in this
		no-difference sense.
		
		\smallskip
		
		For the rest,
		we will assume without loss of generality that $E=\R$; for general inner product space $E$, as in the proof of \cite[Theorem~3.2]{CM25}, one can write $Y^\tau_t(x)$ as a linear combination of a finite number of suitable terms and control the corresponding scalar valued coefficients (see \eqref{eq:Y^tau}).
		
		We now prove by induction on $\mcb{i}(\tau)\eqdef |I_\tau|$, where we recall that $I_\tau$ denotes the set of all inner vertices of $\tau$, that \eqref{eq:inductive-hypothesis} holds for all $p\ge 1$,
		$\beta\in[\alpha,0]$, $\kappa\in[0,1+k\alpha)$, $\delta>-|\tau|/2+\beta/2$, and uniformly in $\e,\bar\e\in(0,1)$ and $t\in(0,1)$.
		
		\medskip
		
		\textbf{Base case.}
		The base case corresponds to elements $\tau$ of $\CT^N$ with $\mcb{i}(\tau)=1$, and the associated trees are
		$\tau= I'(\<Xi>) \prod_{i=1}^k I(\<Xi>), \, \prod_{i=1}^{2k+1} I(\<Xi>)$.
		Since the argument for both cases proceeds in exactly the same way, we only prove the claim for $\tau= I'(\<Xi>) \prod_{i=1}^k I(\<Xi>)$. We thus need to bound the moments of
		\begin{align}
			\label{eq:base-case-terms1}
			&\sup_{|h|_{L^{q_*}}=1}
			\Big|D\CP_t (Z_u^\e-Z_u^{\bar\e})
			\prod_{v\in L_\tau\setminus\{u\}}
			\CP_t Z_v^{\epsilon_v}\Big|_{L^{-n/\beta}}\;,
			\\
			\label{eq:base-case-terms2}
			&\sup_{|h|_{L^{q_*}}=1}
			\Big|\CP_t (Z_u^\e-Z_u^{\bar\e})\,
			D\CP_t Z_w^{\epsilon_w}
			\prod_{v\in L_\tau\setminus\{u,w\}}
			\CP_t Z_v^{\epsilon_v}\Big|_{L^{-n/\beta}}\;,
		\end{align}
		where $u,w\in L_\tau$ are distinct, and
		$Z_{\cdot}\in\{h,X\}$ with exactly one leaf carrying the label $h$.
		Moreover,
		we write
		$Z_{\cdot}^{\epsilon_{\cdot}}
		\eqdef
		\chi^{\epsilon_{\cdot}} * Z_{\cdot}$,
		where
		$ \epsilon_{\cdot}\in\{\e,\bar\e\}$.
		
		We derive the desired moment bounds of the first expression \eqref{eq:base-case-terms1} with $Z_{u}=X$ and $\epsilon_{v}=\e$ for all $v\in L_\tau\setminus \{u\}$; the remaining cases can be estimated analogously, yielding the same bound, and we highlight below the main changes.
		
		Denoting $q\eqdef -n/\beta \ge {q_*}$, we have, by H\"older's inequality,
		\begin{equation}\label{eq:base-case}
			\begin{aligned}
				\E^{\frac{1}{p}}\!\! & \sup_{|h|_{L^{q_*}}= 1} \Big|\CP_t h^\e D\CP_t (X^\e\!-X^{\bar\e}) \prod_{i=1}^{k-1} \CP_t X^\e \Big|^p_{L^{q}}
				\\&
				\lesssim   \E^{\frac{1}{p}} \Big|D\CP_t (X^\e\!-X^{\bar\e}) \prod_{i=1}^{k-1} \CP_t X^\e \Big|_{L^\infty}^p   \sup_{|h|_{L^{q_*}}= 1} |\CP_t h^\e|_{L^{q}} 
				\\
				&\lesssim   \Big(\prod_{i=1}^{k-1} \big\| |\CP_t X^\e|_{L^\infty}\big\|_{L^{kp}}\Big) \big\| |D\CP_t (X^\e\!-X^{\bar\e})|_{L^\infty}\big\|_{L^{kp}} \sup_{|h|_{L^{q_*}}= 1} t^{\frac{n}{2q}-\frac{n}{2{q_*}}}|h^\e|_{L^{q_*}} 
				\\
				&\lesssim 
				t^{\frac{k(\alpha-\varsigma)-1-\kappa}{2}} \big\| \big|X^\e\big|_{\CC^{\alpha-\varsigma}} \big\|_{L^{kp}}^{k-1} \, \big\| \big| X^\e\!-X^{\bar\e}\big|_{\CC^{\alpha-\kappa-\varsigma}} \big\|_{L^{kp}}\, t^{\frac{\alpha-\beta}{2}} \sup_{|h|_{L^{q_*}}= 1} \! |h^\e|_{L^{q_*}}
				\\
				&\lesssim  t^{\frac{k(\alpha-\varsigma)-1-\kappa}{2}}|\e-\bar\e|^{\kappa}\, t^{\frac{\alpha-\beta}{2}}
				=   t^{\frac{(k+1)\alpha-1-\beta-\kappa-k\varsigma}{2}}|\e-\bar\e|^{\kappa}\;, 
			\end{aligned}
		\end{equation}
		where $\varsigma>0$ can be chosen arbitrarily small.
		Above, we applied \eqref{eq:X_eps_diff},
		once with $\ubar{\kappa}=\kappa \in [0,1+k\alpha) \subset [0,1]$ and $\eta=\alpha-\kappa-\varsigma$, and $k-1$ times with $\ubar{\kappa}=0$ and $\eta=\alpha-\varsigma$. (We remark that Lemma~\ref{lem:X^eps_Ceta_moment-bounds} remains valid for the Fourier-truncated approximations, and so does \eqref{eq:X_eps_diff}.)
		To estimate the factor involving $h$, we used \lem{lem:heat-flow-estimates}\ref{item:Leb} and that $q\ge {q_*}$.
		Since $|I'(\<Xi>) \prod_{i=1}^k I(\<Xi>)|=(k+1)\alpha-1$, the final bound obtained in \eqref{eq:base-case} aligns with the desired estimate for the base case. 
		
		The case $Z_u = h$ in \eqref{eq:base-case-terms1} follows analogously, the main change being that we apply
		\begin{equation}\label{eq:h_diff_flow}
			|D^a \CP_t(h^\e - h^{\bar\e})|_{L^{q}}\lesssim 
			|\e - \bar\e|^\kappa t^{\frac{\alpha-\beta-|a|-\kappa}{2}}|h|_{L^{q_*}}\;.
		\end{equation}
		To see this latter bound, by interpolation, it suffices to consider $\kappa=0,1$.
		The case $\kappa=0$ is Lemma~\ref{lem:heat-flow-estimates}\ref{item:Leb},
		while for $\kappa=1$ we use $|f^\e-f^{\bar \e}|_{L^q}\lesssim |\e-\bar\e||D f|_{L^q}$ thus
		\[
		|D^a \CP_t(h^\e - h^{\bar\e})|_{L^{q}}
		\lesssim |\e-\bar\e||D D^a \CP_t h|_{L^{q}}
		\lesssim |\e-\bar\e|t^{\frac{\alpha-\beta-|a|-1}{2}}|h|_{L^{q_*}}
		\]
		where the final bound uses again Lemma~\ref{lem:heat-flow-estimates}\ref{item:Leb}.
		
		The moment bounds for \eqref{eq:base-case-terms2} follow in the same manner, with the minor change that we apply \eqref{eq:h_diff_flow} with $a=0$ in the case $Z_u=h$.
		
		\medskip
		
		\textbf{Induction step.}
		Suppose now that
		the claimed bound \eqref{eq:inductive-hypothesis}
		holds for all $\sigma\in\CT^N$ with $\mcb{i}(\sigma)\le m$, $m \in \N$,
		and consider any $\tau\in\CT^N$ with $\mcb{i}(\tau)= m+1$.
		Since the argument for both cases $\tau = I'(\tau_{k+1}) \prod_{i=1}^k I(\tau_i)$ and $\tau = \prod_{i=1}^{2k+1} I(\tau_i)$ is similar, we only consider trees of the former structure.
		
		We assume further that $\tau^h = I'(\tau_{k+1}^h)\prod_{i=1}^k I(\tau_i)$ and that all the leaves of $\tau_i\neq\tau_{k+1}$ represent either $X^\e$ or $X^{\bar\e}$, and the leaf corresponding to convolution with $\chi^\e-\chi^{\bar\e}$ belongs to $\tau_{k+1}$; the other cases where $h$ or the leaf with $\chi^\e-\chi^{\bar\e}$ are on different branches are handled similarly. (Although there is no $X^\e - X^{\bar\e}$ associated with the leaves of $\tau_i \neq \tau_{k+1}$, we still write $Y^{\tau_i}$ to denote the corresponding factor.) 
		Assume also that $\mcb{i}(\tau_i) \ge 1$ for all $i$; if $\mcb{i}(\tau_i) = 0$ for some $i$, i.e. $\tau_i = \<Xi>$, then
		we obtain the same conclusion without using the inductive hypothesis for $\tau_i$ and instead applying \eqref{eq:X_eps_diff} as in~\eqref{eq:base-case} without a time-integral for this factor.
		
		Letting $p\geq 2$ and applying H\"older's inequality,
		\begin{equs}[eq:proof-of-induction]
			\E^{\frac{1}{p}}\!\!\sup_{|h|_{L^{q_*}}= 1}\!\! |Y^{\tau^h}_{t}|_{L^{q}}^p
			&\lesssim
			\Big(\prod_{i=1}^k \E^{\frac{1}{\tilde{p}}} |\CP_t\star Y^{\tau_i}|_{L^\infty}^{\tilde{p}}\Big) \;
			\E^{\frac{1}{\tilde{p}}}\!\! \sup_{|h|_{L^{q_*}}= 1}\!\! |D\CP_t\star Y^{\tau_{k+1}^h}|_{L^{q}}^{\tilde{p}} \qquad
			\\ & \lesssim
			\Big(\prod_{i=1}^k \int_0^t (t-r)^{\frac{\alpha}{2}}\, \E^{\frac1{\tilde{p}}} |Y^{\tau_i}_{r}|_{\CC^\alpha}^{\tilde{p}} \,\mrd r\Big)
			\\&\quad \times
			\int_0^t (t-r)^{\frac{\alpha-\beta}{2}-\frac{1}{2}}\, \E^{\frac1{\tilde{p}}}\!\!\! \sup_{|h|_{L^{q_*}}= 1}\!\!\! |Y^{\tau_{k+1}^h}_{r}|_{L^{q_*}}^{\tilde{p}} \,\mrd r \qquad
			\\  &\lesssim
			|\e-\bar\e|^{\kappa}
			\Big(\prod_{i=1}^k \int_0^t (t-r)^{\frac{\alpha}{2}}\, r^{\frac{|\tau_i|-\alpha-\varsigma}{2}} \,\mrd r\Big)
			\\& \quad \times
			\int_0^t (t-r)^{\frac{\alpha-\beta}{2}-\frac{1}{2}}\, r^{\frac{|\tau_{k+1}|-\alpha-\kappa-\varsigma}{2}} \,\mrd r \qquad
			\\ 
			[.4em]
			&\lesssim
			t^{\frac{|\tau|-\beta-\kappa-(k+1)\varsigma}{2}}|\e-\bar\e|^{\kappa}\;,\qquad
		\end{equs}
		where $\tilde{p}=(k+1)p$ and, as before, $\varsigma>0$ represents an arbitrarily small quantity.
		In the third bound of \eqref{eq:proof-of-induction}, we applied the no-difference counterpart of 
		\eqref{eq:Kol-C_beta} and the inductive hypothesis 
		as follows. 
		For the factor corresponding to $\tau_i\neq \tau_{k+1}$, we used \eqref{eq:Kol-C_beta}
		(with $\beta^-=\alpha$, $\beta=\alpha+\varsigma/2$, $\kappa=0$, $\delta=-|\tau_i|/2+\alpha/2+\varsigma/2$, and with $p$ therein replaced by $\tilde{p}$), yielding
		\begin{equ}
			\E^{\frac{1}{\tilde{p}}} |Y^{\tau_i}_{r}|_{\CC^\alpha}^{\tilde{p}} \lesssim
			r^{\frac{|\tau_i|-\alpha}{2}-\frac{\varsigma}{2}}\;.
		\end{equ}
		We further applied the inductive hypothesis \eqref{eq:inductive-hypothesis} (with $\beta=\alpha$, $\kappa\in[0,1+k\alpha)$, $\delta=-|\tau_{k+1}|/2+\alpha/2+\varsigma/2$, and with $p$ therein replaced by $\tilde{p}$) to estimate the factor associated with $\tau_{k+1}$.
		In the final step in \eqref{eq:proof-of-induction}, we also used that
		\begin{equ}[eq:sing_criterion]
			(\alpha-\beta-1)/2\,,\; \alpha/2\,,\; ( |\tau_i|-\alpha-\kappa)/2 \;>\; -1
		\end{equ}
		so that the singularities are integrable. 
		For the final bound in \eqref{eq:sing_criterion}, note that $|\tau_i|\geq \big|I'(\<Xi>) \prod_{i=1}^kI(\<Xi>)\big| = (k+1)\alpha-1$ due to $\mcb{i}(\tau_i) \ge 1$ and $|\cdot|$ being increasing in $\noise{\cdot}$ by the assumption $\alpha+1/k>0$ and by Remark~\ref{rem:homogeneity}.
		Thus $( |\tau_i|-\alpha-\kappa)/2 \geq \frac{k\alpha-1-\kappa}{2}>-1$, where the final inequality is due to $\kappa < 1+k\alpha$.
		
		Taking $\varsigma>0$ arbitrarily small, we conclude that \eqref{eq:inductive-hypothesis} holds for $\tau$ as well, which completes the induction step
		and thus the proof of \eqref{eq:Z_eps_diff.}.
		
		\medskip
		
		\textit{Proof of \eqref{eq:Z_time_diff.}.}
		We prove \eqref{eq:Z_time_diff.} by combining the spectral-gap argument
		used in the proof of \eqref{eq:Z_eps_diff.} with the no-difference bounds
		established above.
		We note that using \eqref{eq:SG-for-Y^tau}--\eqref{eq:Kol-C_beta}  with
		$\kappa=0$, one similarly obtains 
		\begin{equ}[eq:Z_eps]
			\| \scal{t^{\delta} X^{\tau,\e}_t , \phi^\lambda_z} \|_{L^p} \lesssim  \lambda^{\beta}\;,
		\end{equ}
		with $\tau\in\CT^N$, $\beta \in [\alpha,0]$, $\delta>-|\tau|/2+\beta/2$, 
		and all $p\geq1$, and uniformly in $t\in(0,1)$, $\e\in(0,1)$, $\lambda\in(0,1]$, and $z\in\T^n$.
		The same estimate holds uniformly for the corresponding
		Fourier-truncated objects. 
		As above, whenever the spectral-gap inequality is invoked, all statements are first formulated for the objects constructed from the Fourier truncations, and we suppress the
		superscript $(R)$. 
		
		For $t>2s$ we can use the no-difference version of \eqref{eq:Z_eps_diff.} with $\kappa=0$ to conclude
		\[
		\|  \scal{t^{\delta+\kappa/2} X_t^{\tau,\e},\phi^\lambda_z}\|_{L^p}+
		\|\scal{s^{\delta+\kappa/2} X_s^{\tau,\e},\phi^\lambda_z} \|_{L^p} \lesssim (t^{\kappa/2} + s^{\kappa/2})\lambda^\beta \lesssim |t-s|^{\kappa/2} \lambda^{\beta}\;,
		\]
		which implies \eqref{eq:Z_time_diff.} by the triangle inequality.
		
		It thus suffices to consider the case $s<t\le2s$.
		Again, since the arguments for the two possible tree structures are
		analogous, we restrict ourselves to trees of the form
		$\tau=I'(\tau_{k+1})\prod_{i=1}^k I(\tau_i)$.
		
		For $\kappa\in[0,1+k\alpha)$, set
		\[
		Z^\e_{t,s}
		\eqdef
		t^{\delta+\kappa/2}X_t^{\tau,\e}
		-
		s^{\delta+\kappa/2}X_s^{\tau,\e}\;,
		\qquad
		\CP_{t,s}\eqdef\CP_t-\CP_s\;,
		\]
		and decompose $Z^\e_{t,s}$ as
		\begin{equs}[eq:time-diff_split.]
			Z^\e_{t,s}
			&=
			s^{\delta+\kappa/2}
			\sum_{j=1}^k
			\big(D\CP_s\star X^{\tau_{k+1},\e}\big)
			\Big(
			\prod_{i<j}\CP_t\star X^{\tau_i,\e}
			\Big)
			\big(\CP_{t,s}\star X^{\tau_j,\e}\big)
			\prod_{i>j}\CP_s\star X^{\tau_i,\e}
			\\
			&\quad
			+
			s^{\delta+\kappa/2}
			\big(D\CP_{t,s}\star X^{\tau_{k+1},\e}\big)
			\prod_{i=1}^k \CP_t\star X^{\tau_i,\e}
			\\
			&\quad
			+
			\bigl(t^{\delta+\kappa/2}-s^{\delta+\kappa/2}\bigr)
			\big(D\CP_t\star X^{\tau_{k+1},\e}\big)
			\prod_{i=1}^k \CP_t\star X^{\tau_i,\e}\;.
		\end{equs}
		For the final summand in \eqref{eq:time-diff_split.}, by the no-difference version of \eqref{eq:Z_eps_diff.} with $\kappa=0$ therein,
		\begin{equation}
			\big|t^{\delta+\kappa/2}-s^{\delta+\kappa/2}\big|
			\big\|
			\scal{X_t^{\tau,\e},\phi_z^\lambda}
			\big\|_{L^p}
			\lesssim
			\big|t^{\delta+\kappa/2}-s^{\delta+\kappa/2}\big|
			\,t^{-\delta}\lambda^\beta
			\lesssim
			|t-s|^{\kappa/2}\lambda^\beta,
		\end{equation}
		where the last inequality uses $s<t\le2s$.

		To establish the claimed bound for the term in the second line of \eqref{eq:time-diff_split.} (the other summands would be estimated in exactly the same way), proceeding as in the proof of \eqref{eq:Z_eps_diff.}, we aim to show that, for all $p\ge 1$, 
		$\beta\in[\alpha,0]$, $\kappa\in[0,1+k\alpha)$, $\delta>-|\tau|/2+\beta/2$, and uniformly in $\e\in(0,1)$ and $s,t\in(0,1)$ with $t\in(s,2s]$,
		\begin{equ}[eq:inductive-hypothesis.]
			\Big\| \sup_{|h|_{L^{q_*}}= 1} |Y^{\tau^h}_{t,s}|_{L^{-n/\beta}}\Big\|_{L^p} \lesssim
			t^{-\delta-\kappa/2}|t-s|^{\kappa/2}\;.	
		\end{equ}
		Here and below, all leaves of trees $\tau,\tau_i$ now represent $X^\e$, and $Y^{\tau^h}_{t,s}$ is constructed analogously to $Y^{\tau^h}_{t}$, except 
		that exactly one occurrence of the heat kernel $G^{(j)}_{t - \cdot}$ (associated with the root) is replaced by $G^{(j)}_{t - \cdot} - G^{(j)}_{s - \cdot}$.
		
		\medskip
		
		\textbf{Trees with one inner vertex.}
		Recall the choices $q=-n/\beta$ and ${q_*}=-n/\alpha$. In a case corresponding to \eqref{eq:base-case}, we have
		\begin{equs}[eq:base-case.]
			\E^{\frac{1}{p}}\!\!\sup_{|h|_{L^{q_*}}= 1}\!\! \big| & \big(\CP_{t} X^\e\big)^{k-1}\CP_{t} h^\e\,  D\CP_{t,s} X^\e \big|^p_{L^{q}}
			\\
			&\lesssim   
			\big(\prod_{i=1}^{k-1}\E^{\frac{1}{kp}}|\CP_{t} X^\e|_{L^\infty}^{kp}\big)
			\,\E^{\frac{1}{kp}}|D\CP_{t,s} X^\e|_{L^\infty}^{kp} \sup_{|h|_{L^{q_*}}= 1} |\CP_t h^\e|_{L^{q}} \qquad
			\nonumber\\
			&\lesssim 
			t^{\frac{(k-1)(\alpha-\varsigma)}{2}}
			|t-s|^{\frac{\kappa}{2}}\, t^{\frac{\alpha-1-\kappa-\varsigma}{2}}\, \E^{\frac{1}{p}} |X^\e|_{\CC^{\alpha-\varsigma}}^{kp}  \, t^{\frac{\alpha-\beta}{2}}\qquad
			\nonumber\\
			[.4em]
			& \lesssim |t-s|^{\frac{\kappa}{2}}\,  t^{\frac{(k+1)\alpha-1-\beta-\kappa-k\varsigma}{2}}\;,\qquad
		\end{equs}
		where \lem{lem:heat-flow-estimates}\ref{item:Hol_diff} together with \eqref{eq:X_eps_diff} (with $\kappa=0$ therein) is employed to bound the term involving the time difference of the heat kernel.
		The other factors are estimated in exactly the same way as in \eqref{eq:base-case}. 
		In the above, we consider the object in which the heat kernels associated with edges labelled by $I$ are evaluated at time $t$. However, since $t\in (s,2s]$ by assumption, the same bound holds if some of these kernels are instead evaluated at time $s$.
		
		The remaining cases associated with $I'(\<Xi>) \prod_{i=1}^k I(\<Xi>)$ would be estimated analogously, thereby completing the proof of \eqref{eq:inductive-hypothesis.} for the case $\tau=I'(\<Xi>) \prod_{i=1}^k I(\<Xi>)$.
		
		\medskip
		
		\textbf{General trees.}
		The proof of \eqref{eq:inductive-hypothesis.} for trees of the form $I'(\tau_{k+1}) \prod_{i=1}^k I(\tau_i)$ with $\tau_i\neq \<Xi>$ proceeds via estimates similar to those in \eqref{eq:proof-of-induction}, adjusted analogously to the simpler case above, as detailed below.
		
		For the corresponding case in which
		$\tau^h = I'(\tau_{k+1}^h) \prod_{i=1}^k I(\tau_i)$,
		the edge labelled by $I'$ represents the increment
		$G^{(j)}_{t-\cdot} - G^{(j)}_{s-\cdot}$, $\; |j|=1$,
		and all other heat kernels are evaluated at time $t$,
		\begin{equ}[eq:proof-of-induction.]
			\E^{\frac{1}{p}}\!\!\sup_{|h|_{L^{q_*}}= 1}\!\! |Y^{\tau^h}_{t,s}|_{L^{q}}^p
			\lesssim
			\Big(\prod_{i=1}^k \E^{\frac{1}{\tilde{p}}} |\CP_t\star Y^{\tau_i}|_{L^\infty}^{\tilde{p}}\Big) \;
			\,\E^{\frac{1}{\tilde{p}}}\!\! \sup_{|h|_{L^{q_*}}= 1}\!\! |D\CP_{t,s}\star Y^{\tau_{k+1}^h}|_{L^{q}}^{\tilde{p}}\;.
		\end{equ}
		The factor associated with $\tau_i\neq\tau_{k+1}$ is estimated using
		the Kolmogorov consequence of \eqref{eq:Z_eps}, with
		$\beta^-=\alpha,
		\,
		\beta=\alpha+{\varsigma}/{2},
		\,
		\delta=
		-{|\tau_i|}/{2}
		+{\alpha}/{2}
		+{\varsigma}/{2}$,
		and with $p$ replaced by $\tilde p$. This yields
		\[
		\left\|
		|Y_r^{\tau_i}|_{\CC^\alpha}
		\right\|_{L^{\tilde p}}
		\lesssim
		r^{\frac{|\tau_i|-\alpha-\varsigma}{2}}\;.
		\]
		
		For the factor corresponding to $\tau_{k+1}$, 
		\begin{equs}
			\E^{\frac{1}{\tilde{p}}}\!\! \sup_{|h|_{L^{q_*}}= 1}\!\! |D\CP_{t,s}\star Y^{\tau_{k+1}^h}|_{L^{q}}^{\tilde{p}}  & \lesssim
			\int_0^{s} \E^{\frac1{\tilde{p}}}\!\!\! \sup_{|h|_{L^{q_*}}= 1}\!\!\! |D\CP_{t-r,s-r}Y^{\tau_{k+1}^h}_{r}|_{L^{q}}^{\tilde{p}} \,\mrd r
			\\ & \quad + 
			\int_s^{t} \E^{\frac{1}{\tilde{p}}}\!\! \sup_{|h|_{L^{q_*}}= 1}\!\! |D\CP_{t-r} Y_r^{\tau_{k+1}^h}|_{L^{q}}^{\tilde{p}} \,\mrd r
			\\ & \lesssim
			\int_0^{s} \!\!|t-s|^{\frac{\kappa}{2}} (s-r)^{\frac{\alpha-\beta-1-\kappa}{2}} \, \E^{\frac1{\tilde{p}}}\!\!\! \sup_{|h|_{L^{q_*}}= 1}\!\!\! |Y^{\tau_{k+1}^h}_{r}|_{L^{q_*}}^{\tilde{p}} \,\mrd r
			\\ & \quad + 
			\int_s^{t} \!(t-r)^{\frac{\alpha-\beta-1}{2}}\, \E^{\frac1{\tilde{p}}}\!\!\! \sup_{|h|_{L^{q_*}}= 1}\!\!\! |Y^{\tau_{k+1}^h}_{r}|_{L^{q_*}}^{\tilde{p}} \,\mrd r
			\\ &\lesssim
			|t-s|^{\frac{\kappa}{2}} t^{\frac{1+|\tau_{k+1}|-\beta}{2}-\frac{\kappa}{2}-\frac{\varsigma}{2}}\;,
		\end{equs}
		where $\varsigma>0$ is sufficiently small. The second inequality follows from Lemma~\ref{lem:heat-flow-estimates}\ref{item:Leb_diff}, and the final step uses \eqref{eq:inductive-hypothesis} (with $\kappa=0$ therein). 	
		We also used that $\kappa\in[0,1+k\alpha)$ and $t\in(s,2s]$.
		
		In conclusion, the above shows the desired bound
		\begin{equ}
			\E^{\frac{1}{p}}\!\!\sup_{|h|_{L^{q_*}}= 1}\!\! |\big(\prod_{i=1}^k\CP_t\star Y^{\tau_i}\big)\, D\CP_{t,s}\star Y^{\tau_{k+1}^h}|_{L^{q}}^p
			\lesssim |t-s|^{\frac{\kappa}{2}} \, t^{\frac{|\tau|-\beta-\kappa-(k+1)\varsigma}{2}}\;
		\end{equ}
		for trees $\tau^h=I'(\tau_{k+1}^h)\prod_{i=1}^kI(\tau_i)$ with $\tau_i\neq\<Xi>$, $i\in[k+1]$.
		The argument for the remaining cases with $\tau=I'(\tau_{k+1}) \prod_{i=1}^k I(\tau_i)$, as well as trees of the form $\tau=\prod_{i=1}^{2k+1} I(\tau_i)$, proceeds in exactly the same way. This completes the proof of \eqref{eq:inductive-hypothesis.}. It then remains to apply Lemma~\ref{lem:X^tau,e_mean-zero} and the spectral gap assumption exactly as in the steps \eqref{eq:SG-for-Y^tau}--\eqref{eq:inductive-hypothesis} 
		in the proof of \eqref{eq:Z_eps_diff.} 
		to obtain \eqref{eq:Z_time_diff.}.
	\end{proof}

	\section{Proof of \theo{thm:main}}
	\label{sec:proof-of-main}
	\begin{proof} [of \theo{thm:main}]
		Recall $\alpha = -n/q_*$ from Definition~\ref{def:homogeneity}.
		Set $N = \floor{1/(k\alpha+1)}$ and, for $\kappa>0$ sufficiently small,
		define $\omega_{\<Xi>}\eqdef \alpha-\kappa$
		and, for all $\tau\in \CT^N$,
		\begin{equ}
			\beta_\tau \eqdef \alpha\;,\qquad 	
			\delta_\tau \eqdef -|\tau|/2+\beta_\tau/2 +\kappa/2\;,  \qquad
			\omega_\tau \eqdef \beta_\tau-2\delta_\tau+2\;.
		\end{equ}
		Note that, by \rem{rem:homogeneity},
		$\omega_{\tau}=\noise{\tau}(\alpha+1/k)-1/k-\kappa$ for all $\tau\in\CT^N_{\<Xi>}$.
		We claim that, for $\kappa$ sufficiently small, $N, \omega_{\<Xi>}, \bmbeta,\bmdelta$ satisfy condition~\eqref{eq:CI}.
		
		To prove the claim, we first note that since $\alpha\in(-1/k,0)$, we have $\beta_\tau\in(-1,0)$  for all $\tau\in\CT^N$. 
		Furthermore, for any $ \tau\in \CT^N$, we have $ \delta_\tau<1$ as it is equivalent to
		\begin{equs}
			-\noise{\tau}\big(\tfrac{1}{2k}+\tfrac{\alpha}{2}\big)+ \tfrac{2k+1}{2k}+\tfrac{\alpha}{2}+ \tfrac{\kappa}{2} < 1 \, 	
			\iff
			( \noise{\tau}-1)\big(\tfrac{1}{k}+{\alpha}\big) > \kappa
		\end{equs}		
		(and the last inequality holds for sufficiently small $\kappa$ because $\noise{\tau}>1$ and $\alpha>-1/k$).
		Next, since $\noise{\tau}\le 1/(k\alpha+1)$ and $\kappa>0$, it is clear that $\noise{\tau}(\alpha+1/k)<1/k+\kappa$, which shows $\omega_\tau < 0$. 
		
		It remains to verify that $\omega>-1/k$ and $2\lambda,\gamma >-2$ for $\omega,\lambda,\gamma$ as in \eqref{eq:CI}.
		Observe that $\omega_{\tau}=\noise{\tau}(\alpha+1/k)-1/k-\kappa$ is increasing in $\noise{\tau}$ (since $\alpha>-1/k$).
		Therefore, $\omega = \omega_{\<Xi>} = \alpha -\kappa > -1/k$, where the inequality holds for $\kappa$ sufficiently small (again, also since $\alpha>-1/k$).
		
		To verify $2\lambda,\gamma>-2$, note that $\omega_{\tau}$ is affine and increasing in $\noise{\tau}$. Hence,
		\begin{equs}
			\min&\Big\{\sum_{i=1}^{k+1}\omega_{\tau_i} \,:\, \tau_i \in\CT_{\<Xi>}^N\,,\; \sum_{i=1}^{k+1}\noise{\tau_i}>N\Big\}
			\\ &\qquad\qquad\ge
			k\alpha+ (N-k+1)({\alpha+1/k})-\tfrac{1}{k}-(k+1)\kappa\;,
			\\
			\min&\Big\{\sum_{i=1}^{2k+1}\omega_{\tau_i} \,:\, \tau_i \in\CT_{\<Xi>}^N\,,\; \sum_{i=1}^{2k+1}\noise{\tau_i}>N\Big\}
			\\ &\qquad\qquad\ge
			2k\alpha+ (N-2k+1)(\alpha+1/k)-\tfrac{1}{k}-(2k+1)\kappa\;.
		\end{equs}
		By the above inequalities, the conditions $\lambda>-1$ and $\gamma>-2$ are satisfied provided
		\[
		N>\frac{-\alpha+(k+1)\kappa}{\alpha+1/k}
		\quad\text{ and }\quad
		N>\frac{-\alpha+(2k+1)\kappa}{\alpha+1/k}\;.
		\]
		Since $N=\floor{1/(k\alpha+1)}$, these inequalities hold for sufficiently small $\kappa$.
		
		In conclusion, the above shows that $N,\omega_{\<Xi>},\bmbeta,\bmdelta$ satisfy condition~\eqref{eq:CI}.
		Define $\init\equiv\init_{\omega_{\<Xi>},\bmbeta,\bmdelta}$ and $\Theta\equiv\Theta_{\omega_{\<Xi>},\bmbeta,\bmdelta}$ as in Definition~\ref{def:norms+init+}. 
		Since $\beta_\tau>-1$ and $\delta_\tau<1$ for all $\tau \in \CT^N$,
		\prop{prop:closable_graph} implies that $(\CI,\Theta)$ is continuously embedded in $\CC^{\omega_{\<Xi>}}$.
		
		Furthermore, the assumptions of \theo{thm:convergence-of-mollifications} hold with our choice of parameters, and therefore $\lim_{\eps\downarrow0} X^\eps = X$ in $(\init,\Theta)$ a.s. and in $L^p(\P)$ for all $p \in [1,\infty)$ (since the only possible limit point of $X^\eps$ is $X$).
		In particular, $X$ a.s. takes values in $\CI$.
		
		\medskip
		
		Take now any $ \theta> 0$ such that 
		$k\omega, \lambda > 2\theta-1$ and $\gamma/2-\theta >-1$.
		Define $M = \sup_{\e\in[0,1]}\Theta(X^\e)$ with $X^0 = X$.
		By \theo{thm:convergence-of-mollifications}, $M$ has finite moments of all orders.
		Since $X\in \init$ a.s., by \theo{thm:lwp}, there exists $\nu>0$ such that, for $T^{\nu} \asymp (1+M)^{-2k}$,
		there exists a unique solution $A$ to \eqref{eq:A_eq}-\eqref{eq:A_ic} on $(0,T)\times\T^n$ given by
		\begin{equ}
			A_t = R_t(X)+\CP_t \star \CS_{\<Xi>}^N X\;,\quad R(X) \in \CB_T
		\end{equ}
		and that has initial condition $X$ in the sense of Theorem~\ref{thm:lwp}, and likewise for $X^\eps$.
		Since $\E M^p < \infty$ for all $p \ge 1$, we obtain the second assertion in \eqref{eq:convergence}.
		
		To prove the convergence in \eqref{eq:convergence}, we write using the triangle inequality
		\begin{equ}[eq:A^eps-A]
			|A_t^\e - A_t|_{\CC^\eta} \le |R_t(X^\e)-R_t(X)|_{\CC^\eta} + |\CP_t \star \CS_{\<Xi>}^N X^\e-\CP_t \star\CS_{\<Xi>}^N X|_{\CC^\eta}\;.
		\end{equ}
		For the first term on the right-hand side of \eqref{eq:A^eps-A}, we know from \theo{thm:lwp} that
		\begin{equ}[eq:R]
			|R_t(X^\e)-R_t(X)|_{\CC^\eta}\leq |R_t(X^\e)-R_t(X)|_{\infty} \leq t^\theta \Theta(X^\e,X)\;.
		\end{equ}
		For the second term, since $\eta<\alpha$, we can ensure $\beta_\tau, \omega>\eta$ for all $\tau\in\CT^N$ by taking $\kappa$ sufficiently small.
		If $\CT^N\neq\emptyset$, set $\delta = \max_{\tau\in\CT^N}\delta_\tau$.
		Then by the triangle inequality and the definition of $\Theta\equiv \Theta_{\omega_{\<Xi>},\bmbeta,\bmdelta}$, we have
		\begin{equ}
			|\CS_{t}^N X^\e-\CS_{t}^N X|_{\CC^\eta}\lesssim t^{-\delta}\Theta( X^\e, X)\;,
		\end{equ}
		from which together with the fact that $\delta<1$, we obtain
		\begin{equ}
			|\CP_t \star \CS^N X^\e-\CP_t \star \CS^N X|_{\CC^\eta}\lesssim t^{1-\delta}\Theta(X^\e, X)\;.
		\end{equ}
		If $\CT^N=\emptyset$, then $\CS_{t}^N X^\e\equiv 0\equiv \CS_{t}^N X$ and the above bound holds trivially.
		We also know, again from the definition of $\Theta$, that
		\begin{equ}
			|X^\e- X|_{\CC^{\omega_{\<Xi>}}}\le \Theta( X^\e, X)\;.
		\end{equ}
		Therefore, since $\eta<\omega_{\<Xi>}$,
		\begin{equ}
			|\CP_t \star \CS_{\<Xi>}^N X^\e-\CP_t \star \CS_{\<Xi>}^N X|_{\CC^\eta}\lesssim \Theta( X^\e, X)\;,
		\end{equ}
		which combined with \eqref{eq:A^eps-A} and \eqref{eq:R} implies
		\begin{equ}
			\sup_{\!t\in[0,T]}|A^\e_t - A_t|_{\CC^\eta} \lesssim \Theta( X^\e, X)\;.
		\end{equ}
		We similarly obtain, using heat flow estimates,
		\begin{equ}
			\sup_{\!t\in(0,T)} t^{-\eta/2}|A^\e_t - A_t|_\infty\lesssim \Theta( X^\e, X)\;.
		\end{equ}
		The claimed a.s.- and $L^p$-convergence, $p\in [1,\infty)$, in \eqref{eq:convergence}
		now follow from the last two inequalities and that $\Theta(X^\e, X)\to 0$ a.s. and in $L^p$.
		The facts that $A$ is smooth on $(0,T)\times \T^n$ and $\lim_{t\downarrow 0} |A_t-X|_{\CC^\eta} = 0$ a.s. are direct consequences of \theo{thm:lwp} since we can take $\omega>\eta$.
		
		Finally, 
		we remark that $T$ in principle depends on $\eta$.
		However, fixing first $\eta<-n/q_*$ and $T=T_\eta$ as above,
		we claim that the convergence \eqref{eq:convergence} in fact holds for this $T$ for every other fixed $\bar\eta<-n/q_*$.
		Indeed, it suffices to consider $\bar\eta\in[\eta,-n/q_*)$.
		We can also assume without loss of generality that $T_{\bar\eta}\leq T$.
		Then
		\[
		\sup_{t\in [T_{\bar\eta},T]}
		t^{-\bar\eta/2}|A^\e_t - A_t|_\infty
		\leq
		T_{\bar\eta}^{\eta/2-\bar\eta/2} \sup_{t\in [T_{\bar\eta},T]} t^{-\eta/2}
		|A^\e_t - A_t|_\infty
		\]
		which converges to $0$ a.s. and in $L^p(\P)$ as $\e\downarrow 0$.
		Similarly so does
		\[
		\sup_{t\in [T_{\bar\eta},T]}|A^\e_t-A_t|_{\CC^{\bar\eta}}
		\leq
		\sup_{t\in [T_{\bar\eta},T]}|A^\e_t-A_t|_{\infty}
		\leq T_{\bar\eta}^{\eta/2}\sup_{t\in [T_{\bar\eta},T]}t^{-\eta/2}|A^\e_t-A_t|_{\infty}\;.
		\]
		This proves the final claim and shows that one $T$ works for all $\eta<-n/q_*$.
	\end{proof}
	
	\appendix
	\section{Spectral gap inequality for random fields}
	
	\subsection{Gaussian fields}
	\label{app:SG}

	We aim to verify Assumption~\ref{ass:field} for a Gaussian field $X$ satisfying Assumption~\ref{ass:GF} below. As a consequence, we show that any field satisfying \cite[Assumption~1.1]{CM25} also satisfies Assumption~\ref{ass:field} with $k=1$.
	The parameter $\gamma>0$ below should be interpreted as $d-2$, where $d$ is the scaling dimension of the Gaussian field.
	
	\begin{assumption}\label{ass:GF}
		Suppose $n\ge 1$ is an integer and $\gamma\in(0,n\wedge 2/k)$.
		Let $X$ be a centred, stationary $E$-valued Gaussian random field on $\T^n$
		whose covariance is represented by a kernel
		$C\in L^1(\T^n)$, i.e.
		\begin{equation}\label{eq:cov_def}
			\Cov(\scal{ X,\phi}, \scal{ X,\psi})
			=
			\scal{\phi,C*\psi} \;,
		\end{equation}
		such that $C(x)=C(-x)$ and
		\begin{equation}\label{eq:C_bound}
			\sup_{x\in\T^n\setminus\{0\}}|x|^\gamma\big|C(x)\big| < \infty\;.
		\end{equation}
	\end{assumption}

	\begin{proposition}\label{prop:GF}
		Suppose $X$ is a stationary Gaussian field with covariance $C$ as in \eqref{eq:cov_def} satisfying the bound \eqref{eq:C_bound} for some $\gamma\in [0,n)$.
		Then $X$ satisfies the spectral gap inequality \eqref{eq:SG-inequality} in the sense of Definition \ref{def:SG} with $L^{s}$-norm for any $s > \frac{2n}{2n-\gamma}$.
		If $\gamma=0$, we can take $s=1$.
	\end{proposition}
	
	For the proof of \prop{prop:GF}, we use the following lemma.
	\begin{lemma}\label{lem:SG-for-Gaussians}
		Let $X$ be a stationary $E$-valued Gaussian field on $\T^n$ with covariance $C$.
		Then for every bounded cylindrical functional $F$
		\begin{equ}\label{eq:SG-appendix_A}
			\Var(F)
			\leq
			\E\Big[\Big\langle\frac{\delta F}{\delta X},C*\frac{\delta F}{\delta X}\Big\rangle\Big]\;.
		\end{equ}
	\end{lemma}
	
	\begin{proof}
		Write $F=\Phi(Y)$
		where $\Phi\in\CC^\infty(\R^N)$ is bounded and $Y=(Y_1,\dots,Y_N)$ is given by $Y_i= \langle X,\phi_i\rangle$ for some $\phi_i\in\CC^\infty(\T^n;E)$.
		Then $Y$ is a Gaussian vector with covariance matrix
		\[
		\Sigma_{ij}
		=
		\operatorname{Cov}(Y_i,Y_j)
		=
		\langle \phi_i,C*\phi_j\rangle\;.
		\]
		By the Gaussian Poincar\'e inequality,
		\[
		\Var(\Phi(Y))
		\le
		\E
		\scal{ \nabla\Phi(Y) ,
			\Sigma
			\nabla\Phi(Y)
		}_{\R^N}
		\;,
		\]
		(see, e.g. \cite[Theorem~3.20]{BoucheronLugosiMassart2013} for the standard Gaussian case, and the general case follows by writing $Y = \mu + MG$, where $\mu=\E Y$, $G$ is a standard Gaussian vector, and $M$ is a matrix satisfying $\Sigma = MM^T$).
		\eqref{eq:SG-appendix_A} follows by writing
		$\frac{\delta F}{\delta X} = \sum_{i=1}^N \partial_i\Phi(Y)\phi_i$.
	\end{proof}

	\begin{proof}[of \prop{prop:GF}]
		By monotonicity of $L^s(\T^n)$-norms,
		it suffices to consider $s \in (\frac{2n}{2n-\gamma},2)$.
		Let $s'$ denote the conjugate exponent, i.e.\ $\frac1s+\frac{1}{s'}=1$. By duality,
		\[
		\Big|\Big\langle \frac{\delta F}{\delta X}, C*\frac{\delta F}{\delta X}\Big\rangle\Big|
		\le
		\Big|\frac{\delta F}{\delta X}\Big|_{L^{s}} \Big|C*\frac{\delta F}{\delta X}\Big|_{L^{s'}}\;.
		\]
		Then, by Young's convolution inequality,
		\[
		\Big|C*\frac{\delta F}{\delta X}\Big|_{L^{s'}}
		\le
		|C|_{L^r}\Big|\frac{\delta F}{\delta X}\Big|_{L^{s}}\,,
		\qquad
		\frac{1}{s}+\frac{1}{r}=1+\frac{1}{s'}=2-\frac{1}{s}\;,
		\]
		where we remark that $r>1$ since $s<2$.
		Therefore, using \eqref{eq:SG-appendix_A},
		$\mathrm{Var}(F)
		\le
		|C|_{L^r}\,
		\mathbb E|\frac{\delta F}{\delta X}|_{L^{s}}^2$.
		The assumption \eqref{eq:C_bound}
		implies $C\in L^r(\T^n)$ for $r<\frac{n}{\gamma}$,
		and the first claim follows from
		$\frac{s}{2(s-1)} = r<\frac{n}{\gamma}\iff	s > \frac{2n}{2n-\gamma}$.
		If $\gamma=0$, we can take $s=1$ and $r=\infty$ above.
	\end{proof}
	
	\begin{lemma}\label{lem:GF-symmetry}
		Consider a centred stationary $E$-valued Gaussian field $X$ on $\T^n$ with covariance kernel $C$ as in \eqref{eq:cov_def} satisfying $C(x)=C(-x)$.
		Then $X\eqdist -X$ and $\CR X \eqdist X$ 
		with reflection $\CR$ defined above Assumption \ref{ass:field}.
	\end{lemma}
	
	\begin{proof}
		$X\eqdist -X$ follows because $X$ is centred,
		while $\CR X$ has covariance $\CR C$, which agrees with $C$ by assumption.
	\end{proof}
	
	\begin{corollary}\label{cor:GF-implies-field}
		Suppose \assu{ass:GF} holds. Then $X$ satisfies \assu{ass:field}.
	\end{corollary}
	
	\begin{proof}
		$X$ satisfies the symmetry conditions in \assu{ass:field} by \lem{lem:GF-symmetry}.
		For every
		$q_* < 2n/\gamma$,
		its conjugate exponent $q^*=q_*/(q_*-1)$ satisfies
		$
		q^*>\frac{2n}{2n-\gamma}
		$
		so \prop{prop:GF} implies that $X$ satisfies the spectral gap inequality with 
		$L^{q^*}$-norm.
		It remains to remark that $nk<2n/\gamma
		\iff
		\gamma<2/k$,
		which holds by \assu{ass:GF},
		so we can take $q_*>nk$ as required in \assu{ass:field}.
	\end{proof}
	
	\subsection{A non-Gaussian example}
	\label{sec:non-Gaussian}
	\begin{example}
		\label{ex:random-phase-field}
		Let $E=\C\simeq\R^2$, viewed as a real inner product space with $\langle z,w\rangle_\C=\Re(z\overline w)$.
		Fix $\gamma\in(0,n\wedge 2/k)$.
		Let $(\theta_m)_{m\in\Z^n\setminus\{0\}}$ be independent random
		variables uniformly distributed on $[0,2\pi)$, and set
		\[
		a_m
		=
		(1+|m|^2)^{-\frac{n-\gamma}{4}}\,,
		\qquad m\in\Z^n\setminus\{0\}\;.
		\]
		Define the $\C$-valued random Fourier field
		\begin{equ}\label{eq:random-phase-field}
			X(x)
			=
			\sum_{m\in\Z^n\setminus\{0\}}
			a_m \mre^{\mri\theta_m}\mre^{2\pi\mri m\cdot x}\;.
		\end{equ}
		Then $X$ satisfies the spectral gap inequality
		\eqref{eq:SG-inequality} with $L^s$-norm for every
		$s>\frac{2n}{2n-\gamma}$.
		Consequently, $X$ satisfies Assumption~\ref{ass:field}.
	\end{example}
	
	\begin{remark}
		Random Fourier series as in \eqref{eq:random-phase-field} are classical and usually referred to as
		\emph{Steinhaus series}, see, e.g.
		\cite[Chapters~1,~5]{Kahane85} and
		\cite[Chapter~I]{MarcusPisier81}, which trace these series back to Paley--Zygmund \cite{Paley_Zygmund_1930}.
		
		The spectral gap argument below is also classical at the level of the
		underlying phases.  The uniform probability measure on $S^1$ satisfies
		the Poincar\'e inequality, and Poincar\'e inequalities are stable under
		products; see \cite[Proposition~4.3.1]{BakryGentilLedoux14}.
		The only point specific to the present example is to express the
		resulting product Dirichlet form in terms of the functional derivative and to bound the resulting Fourier multiplier by an
		$L^s$-norm.
	\end{remark}
	
	\begin{proof}
		Since the coefficients $a_m$ have polynomial decay,
		\eqref{eq:random-phase-field} defines a random element of
		$\CD'(\T^n;\C)$.  Moreover, it is simple to verify that
		\[
		\tau_aX\eqdist X\;,\qquad
		X\eqdist-X\;,\qquad
		\CR X\eqdist X\;,
		\]
		for any $a\in\T^n$, so the symmetry conditions in \assu{ass:field} are satisfied.
		
		Once we show that $X$ satisfies \eqref{eq:SG-inequality} with $L^s$-norm for every $s>\frac{2n}{2n-\gamma}$, it follows that $X$ satisfies \assu{ass:field} by the same argument as in the proof of \cor{cor:GF-implies-field}.
		It thus remains to prove \eqref{eq:SG-inequality} with $L^s$-norm.
		Let $F(\xi)
		=
		\Phi\big(
		\langle\xi,\phi_1\rangle,\ldots,
		\langle\xi,\phi_N\rangle
		\big)$
		be a bounded cylindrical functional.
		For $M\ge1$, let
		\[
		X^M(x)
		=
		\sum_{\substack{m\in\Z^n\setminus\{0\}\\ |m|\le M}}
		a_m \mre^{\mri\theta_m}\mre^{2\pi\mri m\cdot x}\;.
		\]
		The uniform probability measure on the circle satisfies the Poincar\'e
		inequality
		\[
		\Var(f)
		\le
		\int_0^{2\pi}|f'(\theta)|^2
		\frac{\mrd\theta}{2\pi}\;.
		\]
		Tensorisation \cite[Proposition~4.3.1]{BakryGentilLedoux14} therefore yields
		\begin{equ}\label{eq:random-phase-Poincare}
			\Var(F(X^M))
			\le
			\sum_{\substack{m\neq0\\|m|\le M}}
			\E\big|\partial_{\theta_m}F(X^M)\big|^2\;.
		\end{equ}
		Denote
		\[
		v_M
		=
		\left.\frac{\delta F}{\delta\xi}\right|_{\xi=X^M}\;,
		\qquad
		v
		=
		\frac{\delta F}{\delta X}
		\in \CC^\infty(\T^n;\C)\;.
		\]
		Since
		$\partial_{\theta_m}X^M(x)
		=
		\mri a_m \mre^{\mri\theta_m} \mre^{2\pi\mri m\cdot x}
		$,
		the chain rule implies
		\[
		\partial_{\theta_m}F(X^M)
		=
		a_m
		\int_{\T^n}
		\left\langle
		v_M(x),
		\mri \mre^{\mri\theta_m}\mre^{2\pi\mri m\cdot x}
		\right\rangle_\C
		\mrd x\;,
		\]
		Consequently,
		\[
		\big|\partial_{\theta_m}F(X^M)\big|
		\le
		a_m |\widehat v_M(m)|\;,
		\]
		where $\widehat{f}(m) = \int_{\T^n} f(x) \mre^{-2\pi\mri m\cdot x} \mrd x$ denotes the Fourier transform.
		Hence
		\begin{equ}\label{eq:random-phase-Fourier-SG}
			\Var(F(X^M))
			\le
			\E
			\sum_{\substack{m\neq0\\|m|\le M}}
			a_m^2|\widehat v_M(m)|^2\;.
		\end{equ}
		Let $C$ be the periodic distribution with Fourier coefficients
		\[
		\widehat C(0)=0\;,
		\qquad
		\widehat C(m)
		=
		a_m^2
		=
		(1+|m|^2)^{-\frac{n-\gamma}{2}}\,,
		\quad m\neq0\;.
		\]
		Then
		\[
		\sum_{m\neq0}
		a_m^2|\widehat v(m)|^2
		=
		\langle v,C*v\rangle_{L^2(\T^n;\C)}\;.
		\]
		Up to constants and the zero Fourier mode, 
		$C$ is the periodic Bessel kernel of order $n-\gamma$ and satisfies
		\begin{equ}[eq:C_integrable]
			|C(x)|
			\lesssim
			|x|^{-\gamma}\;,
			\qquad x\neq0\;.
		\end{equ}
		This follows, for example, from the representation
		\[
		(1+|m|^2)^{-\frac{n-\gamma}{2}}
		=
		\frac1{\Gamma((n-\gamma)/2)}
		\int_0^\infty
		t^{\frac{n-\gamma}{2}-1}
		\mre^{-t} \mre^{-|m|^2t}\,\mrd t
		\]
		together with the Fourier representation of the periodic heat kernel and the bound
		\[
		G_t(x)
		\lesssim
		t^{-n/2}
		\exp\left(-c\frac{|x|^2}{t}\right)\;,
		\qquad t\in (0,1]\;.
		\]
		We have $X^M\to X$ in $\CD'(\T^n;\C)$ as $M\to\infty$, thus
		$F(X^M)\to F(X)$ and $v_M\to v$ almost surely.
		Moreover, the numbers $\scal{X^M,\phi_i}$ and $\scal{X,\phi_i}$ all lie in a sufficiently large interval in $\R$, uniformly in $M,i$, due to the boundedness of the coefficients $a_m \mre^{\mri\theta_m}$ in \eqref{eq:random-phase-field} and smoothness of $\phi_i$.
		Boundedness of
		$\Phi$ and its derivatives on any ball allows passage to the limit in \eqref{eq:random-phase-Fourier-SG} by dominated convergence, yielding
		\begin{equ}\label{eq:random-phase-SG-K}
			\Var(F(X))
			\le
			\E\langle v,C*v\rangle_{L^2(\T^n;\C)}\;.
		\end{equ}
		Since $C\in L^r(\T^n)$ for every $r<\frac n\gamma$ due to \eqref{eq:C_integrable},
		the rest of the proof proceeds exactly like that of \prop{prop:GF}, with \eqref{eq:random-phase-SG-K} replacing the use of \lem{lem:SG-for-Gaussians}.
	\end{proof}
	
	\section{Embedding and convolution estimates}\label{app:Besov}
	
	Recall the notation from Section~\ref{sec:notation} concerning function spaces $\CC^\eta$.
	
	\begin{lemma}\label{lem:convolution-estimate}
		Let $\eta\leq 0$ and $\kappa\in[0,1]$ such that $\eta+\kappa\leq 0$.
		Let $\chi^\e$ be a standard mollifier on $\T^n$ at scale $\e\in(0,1]$ and define $h^\e = h * \chi^\e$ and denote $h^0=h$. Then, uniformly in $\e,\bar\e\in[0,1]$,
		\begin{equ}
			|h^\e - h^{\bar\e}|_{\CC^\eta}
			\lesssim
			|h|_{\CC^{\eta+\kappa}} |\e-\bar\e|^\kappa\;.
		\end{equ}
	\end{lemma}
	
	\begin{remark}
		Lemma~\ref{lem:convolution-estimate} actually holds for all $\eta\in\R$ and $\kappa\in[0,1]$,
		but we only use it for $\eta+\kappa\le 0$, so we focus on this restricted case for simplicity.
	\end{remark}
	
	\begin{proof}
		The case $\eta=0$ is trivial, so assume $\eta<0$.
		We first prove the translation estimate
		\begin{equation}\label{eq:translation-Holder}
			|\tau_z h-h|_{\CC^\eta}
			\lesssim
			|z|^\kappa |h|_{\CC^{\eta+\kappa}}\;,
		\end{equation}
		where we recall $\tau_z h(x)=h(x-z)$.
		It suffices to consider $|z|\leq 1$.
		
		We recall that in the definition \eqref{eq:CC_eta_def} of $\CC^\eta$, one may equivalently require the test functions to have support in any fixed ball and have any fixed
		number of bounded derivatives sufficiently large compared with $|\eta|$.
		
		Fix an admissible test function $\phi_x^\lambda$,
		$\lambda\in(0,1]$.
		Write
		\begin{equation}\label{eq:h_translation}
			\langle\tau_z h-h,\phi_x^\lambda\rangle
			=
			\langle h,\tau_{-z}\phi_x^\lambda-\phi_x^\lambda\rangle\;.	
		\end{equation}
		If $|z|\le\lambda$, then the test function on the right is a rescaling at
		scale $\lambda$ of $\phi(\,\cdot-z/\lambda)-\phi$, whose admissible test-function
		norm is bounded by a constant times $|z/\lambda|\leq 1$,
		so the right-hand side of \eqref{eq:h_translation} is bounded by $\lesssim \lambda^{\eta+\kappa-1}|z| |h|_{\CC^{\eta+\kappa}}$.
		Moreover, in all cases, $\big| \langle\tau_z h-h,\phi_x^\lambda\rangle\big| \leq 2\lambda^{\eta+\kappa}|h|_{\CC^{\eta+\kappa}}$ by the triangle inequality.
		Therefore
		\[
		\big|
		\langle\tau_z h-h,\phi_x^\lambda\rangle
		\big|
		\lesssim
		\lambda^{\eta+\kappa}
		\min\{1,|z|/\lambda\}
		|h|_{\CC^{\eta+\kappa}}\;.
		\]
		Using $\min\{1,r\} \leq r^\kappa$ for $r\ge0$ and $\kappa\in[0,1]$, we obtain \eqref{eq:translation-Holder}.
		
		We now write
		\begin{equation}\label{eq:mollifier-translation}
			h^\e-h^{\bar\e}
			=
			\int_{\R^n}\chi(y)
			\big(
			\tau_{\e y}h-\tau_{\bar\e y}h
			\big)\,\mrd y
		\end{equation}
		and obtain
		\[
		|h^\e-h^{\bar\e}|_{\CC^\eta}
		\le
		\int_{\R^n}
		|\chi(y)|
		\,
		|\tau_{(\e-\bar\e)y}h-h|_{\CC^\eta}
		\,\mrd y
		\lesssim
		|\e-\bar\e|^\kappa
		|h|_{\CC^{\eta+\kappa}}\;,
		\]
		where we used \eqref{eq:mollifier-translation} and translation invariance of the $\CC^\eta$-norm in the first bound and 
		\eqref{eq:translation-Holder} in the second bound.
	\end{proof}
	
	\begin{lemma}\label{lem:L-HB-embedding}
		Let $p \in [1,\infty)$. Then
		$L^{p}(\T^n) \hookrightarrow \CC^{-n/p}(\T^n)$.
	\end{lemma}
	
	\begin{proof}
		Let $f \in L^p(\T^n)$ and denote $\eta = -n/p$ and $r=-\floor{\eta}+1$.
		Then uniformly in $\phi\in\CB^r$, $x\in\T^n$ and $\lambda\in (0,1]$,
		\[
		|\scal{f,\phi^\lambda_x}|
		\leq
		|f|_{L^p}
		|\phi^\lambda_x|_{L^{p'}(\T^n)}
		\lesssim
		\lambda^{-n/p}|f|_{L^p}\;,
		\]
		where $p'$ is the H\"older conjugate of $p$
		and the first bound follows by H\"older's inequality while the second follows from $|\phi|_{L^{p'}(\R^n)}\lesssim 1$ and the scaling property
		$|\phi^\lambda_x|_{L^{p'}(\T^n)} = \lambda^{-n/p}|\phi|_{L^{p'}(\R^n)}$ due to the small support of $\phi$.
	\end{proof}

	\endappendix

	\subsection*{Acknowledgements}

I.C. acknowledges support from the European Research Council (ERC) via the Starting Grant SQGT 101116964.
H.M. was supported by an Engineering and Physical Sciences Research Council (EPSRC) doctoral studentship.
\medskip

\noindent
For the purpose of open access, the authors have applied a CC BY public copyright licence to any author accepted manuscript arising from this submission.

	\bibliographystyle{Martin}
	\bibliography{./refs}
	
\end{document}